\documentclass[psamsfonts, reqno]{amsart}

\usepackage{amssymb,amsfonts}
\usepackage[all,arc]{xy}
\usepackage{enumerate}
\usepackage{mathrsfs}
\usepackage{slashbox}
\usepackage{geometry,soul,color}
\usepackage{hyperref}
\hypersetup{
    colorlinks=true,
    linkcolor=blue,
    filecolor=blue,
    urlcolor=blue,
    citecolor=red,
}
\allowdisplaybreaks[4]

\newtheorem{theorem}{Theorem}[section]

\newtheorem{lemma}[theorem]{Lemma}

\theoremstyle{definition}

\newtheorem{remark}{Remark}[section]

\numberwithin{equation}{section}

\title[NONHOMOGENEOUS MAGNETIC B\'ERNARD SYSTEM]{Global well-posedness to the Cauchy problem of 2D nonhomogeneous magnetic B\'enard system with large initial data and vacuum}

\date{}
\begin{document}

\maketitle

\centerline{\scshape Jieqiong Liu}
\medskip
{\footnotesize
 \centerline{School of Mathematics and Statistics, Zhengzhou University}

}
\begin{abstract}
This paper establishes the global well-posedness of strong solutions to the nonhomogeneous  magnetic B\'enard system with positive density at infinity in the whole space  $\mathbb{R}^2$. More precisely, we obtain the global existence and uniqueness of strong solutions for general large initial data. Our method relies on dedicate energy estimates and a logarithmic interpolation inequality.

\noindent{\bf{Keywords:}} Magnetic B\'enard system; Strong solutions; Large initial data; Vacuum
\end{abstract}

\section{Introduction and main result}
In this paper, we consider the nonhomogeneous magnetic b\'enard system on the whole space $\mathbb{R}^2$, which reads as follows:
\begin{equation} \label{Benard}
\left\{
\begin{aligned}
&\partial_t\rho + \mathrm{div}(\rho \boldsymbol{u}) = 0, \\
&\partial_t(\rho \boldsymbol{u}) + \mathrm{div} (\rho \boldsymbol{u}\otimes \boldsymbol{u}) - \mu\Delta\boldsymbol{u}+\nabla P = \rho \theta \boldsymbol{e}_2+\boldsymbol{b}\cdot  \nabla \boldsymbol{b},\\
&\partial_t \boldsymbol{b}-\nu \Delta   \boldsymbol{b}+\boldsymbol{u}\cdot\nabla \boldsymbol{b}-\boldsymbol{b}\cdot\nabla \boldsymbol{u}=\boldsymbol{0},\\
& \partial_t (\rho \theta) + \mathrm{div} (\rho \boldsymbol{u} \theta)-\kappa \Delta \theta = \rho  \boldsymbol{u} \cdot \boldsymbol{e}_2, \\
&\mathrm{div} \boldsymbol{u}=\mathrm{div}\boldsymbol{b}= 0,
\end{aligned}
\right.
\end{equation}
with the initial condition
\begin{equation} \label{initial}
(\rho, \rho\boldsymbol{u}, \rho\theta,\boldsymbol{b})(x, 0) = (\rho_0, \rho_0\boldsymbol{u}_0, \rho_0\theta_0,\boldsymbol{b}_{0})(x), \quad \quad x \in \mathbb{R}^{2},
\end{equation}
and the far field behavior (in some weak sense)
\begin{equation} \label{far}
(\rho, \boldsymbol{u}, \theta,\boldsymbol{b})(x, t)\to(\tilde{\rho},\boldsymbol{0},0,\boldsymbol{0})
\quad \quad \quad \text{as}~  |x| \to \infty,\quad t>0,
\end{equation}
for some positive constant $\tilde{\rho}$. Here, the unknown functions $\rho, \boldsymbol{u}=(u^{1},u^{2}), \boldsymbol{b}=(b^{1},b^{2}), P$ and $\theta$ are the fluid density, velocity, magnetic field, pressure, and absolute temperature, respectively. The positive constant $\mu$ stands for the viscosity coefficient, and positive constant $\kappa$ is the heat conducting coefficient, $\boldsymbol{e}_{2}=(0,1)^T$ denotes the vertical unit vector.

The magnetic b\'enard system describes the heat convection phenomenon under the presence of the magnetic field, which plays an important role in engineering and physics. The forcing term $\rho\theta \boldsymbol{e}_{2}$ in the momentum equation $\eqref{Benard}_2$ describes the action of the buoyancy force on fluid motion, and $\rho \boldsymbol{u}\cdot\boldsymbol{e}_{2}$ models the Rayleigh-B\'enard convection in a heated inviscid fluid. Due to their physical importance, wide range of applications and mathematical challenge, the mathematical study of this system has attracted many mathematicians.

When we don't take into account $\eqref{Benard}_3$, that is $\theta \equiv 0$, system \eqref{Benard} reduces to the nonhomogeneous incompressible MHD equations, which have been widely studied. For the case that the initial density has a positive lower bound, Gerbeau and Le Bris \cite{Ger} and Desjardins and Le Bris \cite{Des} studied the global existence of weak solutions of finite energy in the whole space or in the torus, respectively. Lately, Chen et al. \cite{Chen} established a global solution for the initial data belonging to critical Besov spaces (see also \cite{Bie}). On the other hand, in the presence of vacuum, motivated by the work of Choe and Kim \cite{Choe}, Chen et al. \cite{ChenQ} obtained the local existence of strong solutions to the 3D Cauchy problem under the following compatibility condition:
\begin{equation} \label{compa}
-\mu \Delta \boldsymbol{u}_0+\nabla P_0-\boldsymbol{b}_0 \cdot \nabla \boldsymbol{b}_0=\sqrt{\rho_0} \boldsymbol{g}
\end{equation}
for some $\left(P_0, \boldsymbol{g}\right) \in H^1 \times L^2$. With the help of a logarithmic type Sobolev inequality, Huang and Wang \cite{Huang} proved the global existence of strong solutions with general large data on two-dimensional bounded domains under the condition \eqref{compa}. Meanwhile, by using spatial-weighted method, the global large strong solution to the Cauchy problem on the whole space $\mathbb{R}^2$ with zero far field density was established by L\"u et al. \cite{Lv}.

If the motion occurs in the absence of magnetic field, that is, $\boldsymbol{b} \equiv 0$, system \eqref{Benard} reduces to the nonhomogeneous B\'enard system. For the initial density allowing vacuum states, Zhong obtained the global existence and uniqueness of strong solutions to the 2D initial boundary value problem with general large data in \cite{Zhong2} and the global existence and uniqueness of strong solutions to the 3D Cauchy problem for suitable small initial data in \cite{Zhong3}. Recently, by weighted energy method, Zhong \cite{Zhong1, Zhong5} established the local and gloabl existence of strong solutions to the 2D Cauchy problem by assuming that the initial density decays not too slow at infinity. Meanwhile, the global large solution to the 2D Cauchy problem for the case of positive density at infinity was obtained by Li \cite{Li}.

Let's go back to the system \eqref{Benard}. When $\rho$ is a positive constant, which means the fluid is homogeneous, the magnetic B\'enard system has been extensively studied. Zhou et al. \cite{ZhouY} proved the global well-posedness of smooth solutions with zero thermal conductivity. Cheng and Du \cite{Cheng} established the global well-posedness without thermal diffusivity and with vertical or horizontal magnetic diffusion. The global regularity with horizontal dissipation, horizontal magnetic diffusion and with either horizontal or vertical thermal diffusivity was obtained by Ye \cite{Ye}. When the density is not constant, the mathematical analysis of \eqref{Benard} becomes more subtle. Zhong \cite{Zhong6} established the local strong solution to the two-dimensional Cauchy problem with zero far field state. Subsequently, Liu \cite{Liu} extended this local strong solution to global in time for large initial data. However, whether the global strong solution to the 2D Cauchy problem with positive far field density exists or not is still unknown. In fact, the purpose of this paper is to establish the global strong solution to the Cauchy problem \eqref{Benard}-\eqref{far} for general large initial data.

Before formulating our main result, we first explain the notations and conventions used throughout this paper. We denote by
$$\int \cdot ~dx=\int_{\mathbb{R}^{2}}\cdot ~dx.$$
For $1\le r\le \infty$ and $k\in \mathbb{N}$, the Lebesgue and Sobolev spaces are defined in a standard way,
$$L^{r}=L^{r}(\mathbb{R}^{2}),\quad W^{k,r}=W^{k,r}(\mathbb{R}^{2}) ,\quad H^{k}=W^{k,2}.$$

With the above preparation in hand, we now turn to state our main result.
\begin{theorem}\label{main}
For constant $q>2$, assume that the initial data $(\rho_{0} \ge 0, \boldsymbol{u}_{0}, \theta_{0}, \boldsymbol{b}_{0})$ satisfies
\begin{equation} \label{initial 2}
\left\{
\begin{aligned}
&(\rho_{0}-\tilde{\rho}) \in H^{1}\cap W^{1,q}, \\
&\sqrt{\rho_0}\boldsymbol{u}_{0}\in L^{2}, \nabla \boldsymbol{u}_{0}\in L^{2},\mathrm{div} \boldsymbol
{u}_{0}=0,\\
&\sqrt{\rho_0}\theta_{0}\in L^{2}, \nabla \theta_{0}\in L^{2},\\
&\boldsymbol{b}_{0} \in H^{1}, \mathrm{div} \boldsymbol{b}_{0}=0. \\
\end{aligned}
\right.
\end{equation}
Then, the problem \eqref{Benard}-\eqref{far} admits a unique global strong solution $(\rho, \boldsymbol{u}, \theta, \boldsymbol{b})$ such that for any $0<\tau < T< \infty$ and any $2 \le r <q$, it holds that
\begin{equation}\label{ee}
\left\{
\begin{aligned}
& (\rho-\tilde{\rho}) \in C([0,T]; H^{1} \cap W^{1,q}), \rho_{t} \in L^{\infty}(0,T; L^{2} \cap L^{r}), \\
&\nabla \boldsymbol{u}, \nabla \theta, \nabla \boldsymbol{b} \in L^{\infty}(0,T; H^{1}) \cap L^2(0, T; H^2) \cap C([0, T]; H^1), \nabla P \in L^{\infty}(0,T; L^{2}), \\
& \boldsymbol{b} \in L^{\infty}(0,T; H^{2}) \cap L^2(0, T; H^3) \cap C([0, T]; H^2),\\
&\rho \boldsymbol{u}, \rho \theta  \in C ([0,T]; L^{2}), \nabla \boldsymbol{u}_{t}, \nabla\theta _{t}, \nabla\boldsymbol{b}_{t} \in L^{2}(\tau, T; L^{2}), \\
& \sqrt{\rho}\boldsymbol{u}_{t}, \sqrt{\rho }\theta _{t}, \boldsymbol{b}_{t}
\in L^{\infty}(0,T;L^{2})\cap L^{2}(0,T;L^{2}). \\
\end{aligned}
\right.
\end{equation}
\end{theorem}

\begin{remark}
It should be noted that our Theorem \eqref{main} holds for arbitrarily large initial data and initial density with vacuum. In addition, we removed the compatibility condition on the initial data for the global existence of strong solution by deriving some time-weighted estimates.
\end{remark}

We arrange the structure of this paper as follows: In Sec \ref{section 2}, we collect some elementary facts and useful inequalities, which will be used in later analysis. Sec \ref{section 3} is devoted to deriving some necessary \textit{a priori} estimates of different levels. In addition, the proof of Theorem \ref{main} will be given in Sec \ref{section 4}.

\section{Preliminaries} \label{section 2}
In this section, we shall collect some known facts and analytic inequalities that will be used extensively in the later analysis.

We start with the local existence and uniqueness of strong solution to the problem \eqref{Benard}-\eqref{far} with initial data satisfying \eqref{initial 2}, whose proof can be performed by using standard procedures as in \cite{Zhong6, Fan}
\begin{lemma}\label{local}
Assume that the initial data $(\rho_{0},\boldsymbol{u}_{0},\theta_{0},\boldsymbol{b}_{0})$ satisfies \eqref{initial 2}. Then there exist a small time $T_{0}>0$ and a unique strong solution $(\rho,\boldsymbol{u},\theta,\boldsymbol{b}, P)$ to the problem $\eqref{Benard}$-$\eqref{far}$
in $\mathbb{R}^2 \times (0,T_{0}]$.
\end{lemma}

Next, the following Gagliardo-Nirenberg inequality will be used later frequently, whose proof can be found in \cite{Nirenberg}.
\begin{lemma}\label{kk}
For $ p\in [2, \infty), ~r \in (2, \infty)$, and $s\in (1, \infty)$, there exists some generic constant $C>0$ which may only depend on $p, ~r$, and $s$ such that for $f \in H^{1}(\mathbb{R}^2)$ and $g \in L^{s}(\mathbb{R}^2)\cap D^{1,r}(\mathbb{R}^2)$, we have
\begin{equation}\label{R1}
\begin{aligned}
& \| f \| _{L^{p}}^{p} \le C \| f \|_{L^{2}}^{2} \| \nabla f \|_{L^{2}}^{p-2}, \\
& \| g \| _{L^{\infty}} \le C \| g \|_{L^{s}}^{\frac{s(r-2)}{2r+s(r-2)}} \| \nabla g \|_{L^{r}}^{\frac{2r}{2r+s(r-2)} }.
\end{aligned}
\end{equation}
\end{lemma}

In deriving the higher a priori estimates, the classical regularity results for the Stokes system in the whole space $\mathbb{R}^2$ (see \cite{He}) will play an important role, which can be stated as follows.
\begin{lemma}\label{lem 2.3}
Suppose that $\mathbf{F} \in L^r$ with $r \in (1, \infty)$. Let $(\boldsymbol{u}, P) \in H^1 \times L^2$ be the unique weak solution to the following Stokes problem
\begin{equation} \label{stokes}
\left\{
\begin{aligned}
&-\Delta \boldsymbol{u} + \nabla P = \mathbf{F},  & x \in \mathbb{R}^{2}, \\
&\mathrm{div} \boldsymbol{u} =0,  & x \in \mathbb{R}^{2}, \\
&\boldsymbol{u}(x)\to 0, & |x| \to \infty, \\
\end{aligned}
\right.
\end{equation}
then $(\nabla^2 \boldsymbol{u}, \nabla P) \in L^r$ and there exists a positive constant $C$ depending only on $\mu$ and $r$ such that
\begin{equation} \label{0001}
\left \| \nabla ^{2} \boldsymbol{u} \right \|_{L^{r} }+\left \| \nabla P \right \|_{L^{r} }
\le C\left \| \mathbf{F} \right \|_{L^{r} }.
\end{equation}
\end{lemma}

As pointed out in many papers, the difficulty of two-dimensional fluid dynamic equations in the unbounded domains comes from the critical spatial dimension. Therefore, it seems difficult to obtain directly the $L^p$-norm of the velocity just in terms of $\|\sqrt{\rho}\boldsymbol{u}\|_{L^2}$ and $\|\nabla \boldsymbol{u}\|_{L^2}$. In our setting, thanks to the positivity of the far filed state of density $\tilde{\rho}$, we can obtain the following improved Poincar\'e-type inequality, which will be crucially important. And we refer the readers to \cite[Lemma 2.4]{Li} for the proof.
\begin{lemma}\label{lem 2.4}
For $\tilde{\rho}>0$, assume $(\rho-\tilde{\rho}) \in L^{2}$ with $\rho \ge 0$ and $\sqrt{\rho} v, \nabla v \in L^{2}$, where $v$ is a scalar function or vector function. Then, we have
\begin{equation}\label{R3}
\left \| v \right \|_{L^{2}}^{2}  \le C(\tilde{\rho },\left \| \rho - \tilde
{\rho }\right \|_{L^{2}}  )(\left \| \sqrt{\rho } v \right \|_{L^{2}}^{2}+\left \|
\nabla v \right \|_{L^{2}}^{2} ).
\end{equation}
Furthermore, for $2 \le p < \infty$, we have
\begin{equation}
\left \| v \right \|_{L^{p}}\le C\left \| \sqrt{\rho } v \right \|_{L^{2}}^{\frac{2}{p}}
\left \| \nabla v \right \|_{L^{2}}^{1-\frac{2}{p}}+\left \| \nabla v \right \|_{L^{2}},
\end{equation}
or
\begin{equation}\label{R4}
\left \| v \right \|_{L^{p}}\le C(\left \| \sqrt{\rho } v \right \|_{L^{2}}+\left \|
 \nabla v \right \|_{L^{2}}).
\end{equation}
Here and in what follows, we use C(f) to emphasize the dependence on the quantity  f.
\end{lemma}

Finally, we end with the following logarithmic interpolation inequality due to Desjardins \cite{Desjardins}.
\begin{lemma}\label{lem 2.5}
Suppose that $ 0 \le \rho \le \bar{\rho}$ and $\boldsymbol{u} \in H^{1}$, then we have
\begin{equation} \label{DA}
\left \| \sqrt{\rho } \boldsymbol{u} \right \|_{L^{4}}^{2} \le C(\bar{\rho})(1+\left \| \sqrt{\rho}
\boldsymbol{u} \right \|_{L^{2}} )\left \|  \boldsymbol{u} \right \|_{H^{1}}\sqrt{\log (2+\left \|
\boldsymbol{u} \right \|_{H^{1}}^{2})}.
\end{equation}
\end{lemma}

\section{\textit{A PRIORI} estimates} \label{section 3}
In this section, we will establish some necessary \textit{a priori} bounds of different levels for strong solution $(\rho, \boldsymbol{u}, \theta, \boldsymbol{b})$ to the problem \eqref{Benard}-\eqref{far} to extend the local strong solution to global-in-time. To this end, let $T>0$ be a fixed time and $(\rho, \boldsymbol{u}, \theta, \boldsymbol{b})$ be the strong solution to \eqref{Benard}-\eqref{far} on $\mathbb{R}^{2}\times (0,T]$ with initial data $(\rho_{0},\boldsymbol{u}_{0},\theta_{0},\boldsymbol{b}_{0})$ satisfying $\eqref{initial 2}$.

Before proceeding, we rewrite an equivalent form of system $\eqref{Benard}$ if we assume that the solution $(\rho,\boldsymbol{u},\theta,\boldsymbol{b})$ is regular enough, which reads as follows:
\begin{equation}\label{Benard 2}
\left\{
\begin{aligned}
&\rho_{t}+\boldsymbol{u}\cdot \nabla \rho=0,\\
&\rho \boldsymbol{u}_{t}+\rho \boldsymbol{u}\cdot \nabla \boldsymbol{u}-\mu \Delta
\boldsymbol{u} +\nabla P=\rho \theta \boldsymbol{e}_{2} + \boldsymbol{b}\cdot \nabla \boldsymbol{b},\\
&\boldsymbol{b}_t- \nu \Delta \boldsymbol{b} + \boldsymbol{u}\cdot\nabla \boldsymbol{b}=\boldsymbol{b}\cdot\nabla \boldsymbol{u},\\
&\rho \theta_{t} + \rho \boldsymbol{u}\cdot \nabla \theta -\kappa \Delta
\theta   = \rho  \boldsymbol{u} \cdot \boldsymbol{e}_2,\\
&\mathrm{div} \boldsymbol{u}=\mathrm{div}\boldsymbol{b}= 0 .
\end{aligned}
\right.
\end{equation}

Throughout this section, we shall use the convention that $C$ denotes a generic positive constant, which may depend on $\tilde{\rho}, \mu, \kappa, q, T$ and initial data.

We start with the following standard energy estimate and the $L^p$-norm estimate of the density.
\begin{lemma}\label{lem 3.1}
It holds that for $0 \le T < T^{*}$,
\begin{align}
& \sup_{[0, T]} \|\rho \|_{L^{\infty }} \le \|\rho_{0}\|_{L^{\infty}}, \label{R7} \\
& \sup_{[0, T]} \|\rho - \tilde{\rho} \|_{L^{p}} \le \|\rho_{0}-\tilde{\rho} \|_{L^{p}} \quad \text{for}~2 \le p \le q, \label{z2}
\end{align}
and
\begin{equation}\label{aa}
\begin{aligned}
\sup_{[0, T]} (\left \| \sqrt{\rho }\boldsymbol{u}  \right \|_{L^{2}}^{2}+
\left \| \sqrt{\rho }\theta  \right \|_{L^{2}}^{2} + \left \|\boldsymbol{b}
\right \|_{L^{2}}^{2}) +\int_{0}^{T}(\left \| \nabla \boldsymbol{u} \right \|_{L^{2}}^{2}+
\left \| \nabla\theta \right \|_{L^{2}}^{2}+\left \| \nabla \boldsymbol{b} \right \|_{L^{2}}^{2})
dt \le C .
\end{aligned}
\end{equation}
\end{lemma}
\begin{proof}
Since $\eqref{Benard 2}_1$ is a transport equation, it is easy to obtain \eqref{R7}. Moreover, $\eqref{Benard 2}_1$ along with $\rho_0 \ge 0$ gives
\begin{equation} \label{0002}
\inf_{\mathbb{R}^2 \times [0, T]} \rho(x, t) \ge 0.
\end{equation}
For positive constant $\tilde{\rho}$, we obtain from $\eqref{Benard 2}_1$ that $\rho-\tilde{\rho}$ satisfies a transport equation
\begin{equation}\label{R8}
(\rho-\tilde{\rho})_{t} + \boldsymbol{u} \cdot \nabla (\rho-\tilde{\rho})= 0.
\end{equation}
Multiplying \eqref{R8} by $p |\rho-\tilde{\rho}|^{p-2}(\rho-\tilde{\rho})$ and integrating over $\mathbb{R}^2$ imply the desired \eqref{z2}.
Multiplying $\eqref{Benard 2}_2$ by $\boldsymbol{u}$, $\eqref{Benard 2}_3$ by $\boldsymbol{b}$, $\eqref{Benard 2}_4$ by $\theta$, and integrating by parts, we obtain from the H\"older inequality and the Cauchy inequality that
\begin{equation}\label{R9}
\begin{aligned}
&\quad\frac{1}{2}\frac{d}{dt} \int(\rho |\boldsymbol{u}|^{2} + \rho \theta^{2} + |\boldsymbol{b}|^{2}) dx+ \int (\mu |\nabla \boldsymbol{u}|^{2} +
\kappa |\nabla\theta|^{2} +\nu |\nabla\boldsymbol{b}|^{2}) dx \\
&= 2 \int \rho \theta (\boldsymbol{u}\cdot \boldsymbol{e}_{2})dx
\le  2\left \| \sqrt{\rho } \boldsymbol{u} \right \|_{L^{2}}
\left \| \sqrt{\rho } \theta\right \|_{L^{2}}
\le  \| \sqrt{\rho } \boldsymbol{u} \|_{L^{2}} ^{2}
+ \| \sqrt{\rho } \theta \|_{L^{2}}^{2},
\end{aligned}
\end{equation}
which together with Gronwall's inequality gives
\begin{equation}\label{R10}
\begin{aligned}
\sup_{[0, T]} (\| \sqrt{\rho } \boldsymbol{u}\|_{L^{2}}^{2}+
\| \sqrt{\rho } \theta \|_{L^{2}}^{2}+ \| \boldsymbol{b} \|_{L^{2}}^{2})
+ \int_{0}^{T}\left(\| \nabla\boldsymbol{u} \| _{L^{2}}^{2}+
\| \nabla\theta \| _{L^{2}}^{2}+ \| \nabla\boldsymbol{b} \|_{L^{2}}^{2} \right) dt \le C.
\end{aligned}
\end{equation}
This completes the proof of Lemma \ref{lem 3.1}.
\end{proof}
\begin{remark}
We can get from \ref{lem 2.4} and \eqref{aa} that
\begin{equation}\label{pp3}
\begin{aligned}
\|\boldsymbol{u}\|_{H^{1}}=\|\boldsymbol{u}\|_{L^{2}}+\|\nabla\boldsymbol{u}\|_{L^{2}}
&\le C(\|\sqrt{\rho\boldsymbol{u}}\|_{L^{2}}
+\|\nabla\boldsymbol{u}\|_{L^{2}})+C\|\nabla\boldsymbol{u}\|_{L^{2}}\\
&\le C\|\nabla\boldsymbol{u}\|_{L^{2}}+C.
\end{aligned}
\end{equation}
\end{remark}
Next, the following lemma concerns the key estimate on the $L^{\infty}(0, T; L^2)$-norm of the gradients of the velocity, magnetic field and absolute temperature.
\begin{lemma}\label{lem 3.2}
It holds that for all $0\le T<T^{*}$ and each integer $i \in \left\{0,1\right\},$
\begin{equation}\label{R11}
\begin{aligned}
\sup_{[0, T]} t^{i} (\|\nabla \boldsymbol{u}\|_{L^{2}}^{2}&+ \|\nabla \theta\|_{L^{2}}^{2} + \| \nabla \boldsymbol{b} \|_{L^{2}}^{2}) \\
& + \int_{0}^{T} t^{i} \left( \| \sqrt{\rho } \boldsymbol{u} _{t} \|_{L^{2}}^{2}
+ \| \sqrt{\rho} \theta _{t} \|_{L^{2}}^{2} + \|\nabla^2 \boldsymbol{b}\|_{L^{2}}^{2}
+ \| |\boldsymbol{b}||\nabla \boldsymbol{b}|\|_{L^{2}}^{2} \right)dt \le C.
\end{aligned}
\end{equation}
\end{lemma}
\begin{proof}
1. Multiplying $\eqref{Benard 2}_2$ by $\boldsymbol{u}_{t}$ and integrating by parts over $\mathbb{R}^{2}$, we get
\begin{equation}\label{R12}
\begin{aligned}
&\quad\frac{\mu }{2 }\frac{d}{dt} \int |\nabla\boldsymbol{u}|^{2}dx+\int\rho|\boldsymbol{u}_{t}|^{2} dx\\&=-\int\rho\boldsymbol{u}\cdot\nabla\boldsymbol{u}\cdot\boldsymbol{u}_{t}dx
+\int\boldsymbol{b}\cdot\nabla \boldsymbol{b}\cdot  \boldsymbol{u}_{t}dx+\int \rho \theta (\boldsymbol{e_{2}}\cdot \boldsymbol{u}_{t})dx.
\end{aligned}
\end{equation}
By H\"older's inequality and the Gagliardo-Nirenberg inequality, we have
\begin{equation} \label{0003}
\begin{aligned}
\left|-\int \rho\boldsymbol{u}\cdot\nabla \boldsymbol{u}\cdot\boldsymbol{u}_{t}dx \right| \le & \|\sqrt {\rho}\boldsymbol{u}_{t}\|_{L^{2}}\|\sqrt{\rho}\boldsymbol{u}\|_{L^{4}}\|\nabla \boldsymbol{u}\|
_{L^{4}}\\
\le & \frac{1}{4}\|\sqrt{\rho}\boldsymbol{u}_{t}\|_{L^{2}}^{2} + C \|\sqrt{\rho}\boldsymbol{u}\|_{L^{4}}^{2} \|\nabla\boldsymbol{u}\|_{L^{4}}^{2}\\
\le & \frac{1}{4}\|\sqrt{\rho}\boldsymbol{u}_{t}\|_{L^{2}}^{2}+ C\|\sqrt{\rho}\boldsymbol{u}\|_{L^{4}}^{2} \|\nabla \boldsymbol{u}\|_{L^{2}} \|\nabla^{2} \boldsymbol{u}\|_{L^{2}}.
\end{aligned}
\end{equation}
In a similar way, we deduce from \eqref{aa} that
\begin{equation} \label{0004}
\begin{aligned}
\left|\int\rho\theta(\boldsymbol{e_{2}}\cdot\boldsymbol{u}_{t})dx\right|
 \le & \|\sqrt{\rho}\boldsymbol{u}_{t}\|_{L^{2}}\|\sqrt{\rho}\theta\|_{L^{2}}\\
\le & \frac{1}{4}\|\sqrt{\rho}\boldsymbol{u}_{t}\|_{L^{2}}^{2} +
\|\sqrt{\rho}\theta\|_{L^{2}}^{2}\\
\le & \frac{1}{4}\|\sqrt{\rho}\boldsymbol{u}_{t}\|
_{L^{2}}^{2}+C.
\end{aligned}
\end{equation}
By the Gagliardo-Nirenberg inequality and $\eqref{Benard 2}_5$ yields
\begin{equation} \label{0005}
\begin{aligned}
\int\boldsymbol{b}\cdot\nabla \boldsymbol{b}\cdot  \boldsymbol{u}_{t}dx &=-\frac{d}{dt}\int\boldsymbol{b}\cdot\nabla\boldsymbol{u}\cdot
\boldsymbol{b}dx+\int\boldsymbol{b}_{t}\cdot\nabla\boldsymbol{u}\cdot\boldsymbol
{b}dx+\int\boldsymbol{b}\cdot
\nabla\boldsymbol{u}\cdot\boldsymbol{b}_{t}dx\\
&=-\frac{d}{dt}\int\boldsymbol{b}\cdot \nabla\boldsymbol{u}\cdot\boldsymbol{b}dx+\int(\nu\Delta
\boldsymbol{b}-\boldsymbol{u}\cdot \nabla\boldsymbol{b}-\boldsymbol{b}\cdot \nabla\boldsymbol{u})
\cdot \nabla\boldsymbol{u}\cdot \boldsymbol{b}dx\\
&\quad+\int\boldsymbol{b}\cdot  \nabla\boldsymbol{u}
\cdot (\nu\Delta\boldsymbol{b}-\boldsymbol{u}\cdot \nabla\boldsymbol{b}-\boldsymbol{b}\cdot \nabla
\boldsymbol{u}) dx\\
&\le -\frac{d}{dt}\int\boldsymbol{b}\cdot \nabla\boldsymbol{u}\cdot\boldsymbol{b}dx+ C \|\Delta
\boldsymbol{b}\|_{L^{2}}\|\boldsymbol{b}\cdot\nabla\boldsymbol{u} \|_{L^{2}}\\
&\quad+ C \|\boldsymbol{u}\|
_{L^{\infty }}\|\nabla\boldsymbol{b}\|_{L^{2}}\| \nabla\boldsymbol{u}\|_{L^{4}}\|\boldsymbol{b}\|
_{L^{4}}+C\|\boldsymbol{b}\cdot\nabla\boldsymbol{u} \|_{L^{2}}^{2}\\
&\le -\frac{d}{dt}\int\boldsymbol{b}\cdot \nabla\boldsymbol{u}\cdot\boldsymbol{b}dx+\frac{\nu}{4}
\|\Delta\boldsymbol{b}\|_{L^{2}}^{2}+C\|\boldsymbol{b}\cdot\nabla\boldsymbol{u} \|_{L^{2}}^{2}\\
&\quad+C\|\boldsymbol{b}\|_{L^{4}} \|\boldsymbol{u}\|_{L^{4}}^{\frac{1}{2}}\|\nabla \boldsymbol{u}\|_{L^{4}}^{\frac{1}{2}} \|\nabla\boldsymbol{b} \|_{L^{2}} \|\nabla \boldsymbol{u}\|_{L^{2}}^{\frac{1}{2}}\|\nabla^2 \boldsymbol{u}\|_{L^{2}}^{\frac{1}{2}}
\\
&\le-\frac{d}{dt}\int\boldsymbol{b}\cdot \nabla\boldsymbol{u}\cdot\boldsymbol{b}dx+\frac{\nu}{4}
\|\Delta\boldsymbol{b}\|_{L^{2}}^{2}+C\|\boldsymbol{b}\|_{L^{6}}^{6}+C\|\nabla \boldsymbol{u}\|
_{L^{3}}^{3}\\
&\quad+C \|\boldsymbol{b}\|_{L^4} \|\boldsymbol{u}\|_{H^1}^\frac{1}{2}\|\nabla \boldsymbol{b}\|_{L^{2}} \|\nabla \boldsymbol{u}\|_{L^{2}}^{\frac{3}{4}} \|\nabla^2 \boldsymbol{u}\|_{L^{2}}^{\frac{3}{4}} \\
&\le-\frac{d}{dt}\int\boldsymbol{b}\cdot \nabla\boldsymbol{u}\cdot\boldsymbol{b}dx+\frac{\nu}{4}
\|\Delta\boldsymbol{b}\|_{L^{2}}^{2}+C\|\boldsymbol{b}\|_{L^{2}}^{2}\|\nabla\boldsymbol{b} \|_{L^{2}}^{4} \\
&\quad +C\|\nabla\boldsymbol{u}\|_{L^{2}}^{2}\|\nabla ^{2}\boldsymbol{u}\|_{L^{2}} + C \|\boldsymbol{b}\|_{L^4} \|\boldsymbol{u}\|_{H^1}^\frac{1}{2}\|\nabla \boldsymbol{b}\|_{L^{2}} \|\nabla \boldsymbol{u}\|_{L^{2}}^{\frac{3}{4}} \|\nabla^2 \boldsymbol{u}\|_{L^{2}}^{\frac{3}{4}}.
\end{aligned}
\end{equation}
where we have used the following Gagliardo-Nirenberg inequality
\begin{align*}
\| \boldsymbol{u}\|_{L^{\infty}} \le C \|\boldsymbol{u}\|_{L^4}^{\frac{1}{2}}\|\nabla \boldsymbol{u}\|_{L^4}^{\frac{1}{2}}, \quad \|\nabla\boldsymbol{u}\|_{L^{4}} \le C \|\nabla \boldsymbol{u}\|_{L^2}^{\frac{1}{2}}\|\nabla^2 \boldsymbol{u}\|_{L^2}^{\frac{1}{2}}.
\end{align*}
Inserting \eqref{0003}-\eqref{0005} into \eqref{R12}, we get from \eqref{aa} that
\begin{equation} \label{R12'}
\begin{aligned}
& \frac{d}{dt}\int \left(\frac{\mu}{2}|\nabla\boldsymbol{u}|^{2} + \boldsymbol{b}\cdot\nabla
\boldsymbol{u}\cdot\boldsymbol{b} \right) dx +\frac{1}{2} \int \rho|\boldsymbol{u}_{t}|^{2} dx \\
\le &  C(\|\sqrt{\rho}\boldsymbol{u}\|_{L^{4}}^{2}+\|\nabla\boldsymbol{u}\|_{L^{2}}) \|\nabla\boldsymbol{u}\|_{L^{2}}\|\nabla^{2}\boldsymbol{u}\|_{L^{2}} +\frac{\nu}{4}
\|\Delta\boldsymbol{b}\|_{L^{2}}^{2}\\
&+ C\|\nabla\boldsymbol{b}\|_{L^{2}}^{4}+C\|\boldsymbol{b}\|_{L^4} \|\boldsymbol{u}\|_{H^1}^\frac{1}{2}\|\nabla \boldsymbol{b}\|_{L^{2}} \|\nabla \boldsymbol{u}\|_{L^{2}}^{\frac{3}{4}} \|\nabla^2 \boldsymbol{u}\|_{L^{2}}^{\frac{3}{4}}.
\end{aligned}
\end{equation}
2. Multiplying $\eqref{Benard 2}_{4}$ by $\theta_{t}$ and integrating over $\mathbb{R}^{2}$ lead to
\begin{equation}\label{R13}
\begin{aligned}
&\quad\frac{\kappa}{2}\frac{d}{dt}\int|\nabla\theta|^{2}dx+\int\rho |\theta _{t}|^{2}dx=-\int \rho \boldsymbol{u}
\cdot\nabla\theta\cdot\theta _{t}dx+\int \rho (\boldsymbol{u}\cdot\boldsymbol{e}_{2})\theta_{t}dx.
\end{aligned}
\end{equation}
By virtue of H\"older's inequality, the Gagliardo-Nirenberg inequality and \eqref{aa}, one has
\begin{equation} \label{0007}
\begin{aligned}
\left|\int\rho\boldsymbol{u}\cdot\nabla\theta\cdot\theta_{t} dx\right|
\le & \|\sqrt{\rho}\theta_{t}\|_{L^{2}}
\|\sqrt{\rho}\boldsymbol{u}\|_{L^{4}}\|\nabla \theta\|_{L^{4}} \\
\le & \frac{1}{4}\|\sqrt{\rho}\theta_{t}
\|_{L^{2}}^{2}+C\|\sqrt{\rho}\boldsymbol{u}\|_{L^{4}}^{2}\|\nabla \theta\|_{L^{2}}\|\nabla^{2} \theta\|
_{L^{2}},
\end{aligned}
\end{equation}
and
\begin{equation} \label{0008}
\begin{aligned}
\left|\int \rho (\boldsymbol{u}\cdot\boldsymbol{e}_{2})\theta_{t}dx\right|
\le & \|\sqrt{\rho}\theta_{t}\|
_{L^{2}}\|\sqrt{\rho}\boldsymbol{u}\|_{L^{2}}\\
\le & \frac{1}{4} \|\sqrt{\rho}\theta_{t}\|_{L^{2}}^{2}+C\|\sqrt{\rho}
\boldsymbol{u}\|_{L^{2}}^{2} \\
\le & \frac{1}{4} \|\sqrt{\rho}\theta_{t}\|_{L^{2}}^{2}+C.
\end{aligned}
\end{equation}
Substituting \eqref{0007} and \eqref{0008} into \eqref{R13}, we obtain
\begin{equation}\label{R13'}
\begin{aligned}
\frac{\kappa}{2}\frac{d}{dt}\int|\nabla\theta|^{2}dx+\frac{1}{2}\int\rho |\theta _{t}|^{2}dx \le C\|\sqrt{\rho}
\boldsymbol{u}\|_{L^{4}}^{2} \|\nabla \theta\|_{L^{2}}\|\nabla^{2} \theta\|_{L^{2}} + C.
\end{aligned}
\end{equation}
3. Multiplying $\eqref{Benard 2}_{3}$ by $\Delta\boldsymbol{b}$ and integrating over $\mathbb{R}^{2}$, it follows from H\"older's inequality and the Gagliardo-Nirenberg inequality that
\begin{equation}\label{R14}
\begin{aligned}
&\frac{1}{2} \frac{d}{dt}\int|\nabla\boldsymbol{b}|^{2} dx+\nu \|\Delta\boldsymbol{b}\|_{L^{2}}^{2}\\
\le & C \int |\nabla\boldsymbol{u}||\nabla\boldsymbol{b}|^2dx + C \int|\boldsymbol{b}||\nabla\boldsymbol{u}||\Delta\boldsymbol{b}|dx \\
\le & C\|\nabla\boldsymbol{u}\|_{L^{3}}\|\nabla\boldsymbol{b}\|_{L^{3}}^{2} +C\|\nabla\boldsymbol{u}\|
_{L^{3}}\|\boldsymbol{b}\|_{L^{6}}\|\Delta\boldsymbol{b}\|_{L^{2}}\\
\le & C\|\nabla\boldsymbol{u}\|_{L^{3}}\|\nabla\boldsymbol{b}\|_{L^{2}}^{\frac{4}{3}}\|\Delta
\boldsymbol{b}\|_{L^{2}}^{\frac{2}{3}}+\|\nabla\boldsymbol{u}\|_{L^{3}}
\|\boldsymbol{b}\|_{L^{6}}
\|\Delta\boldsymbol{b}\|_{L^{2}}\\
\le & \frac{\nu }{4} \|\Delta\boldsymbol{b}\|_{L^{2}}^{2}+C\|\nabla\boldsymbol{u}\|_{L^{3}}^{\frac{3}{2}}
\|\nabla\boldsymbol{b}\|_{L^{2}}^{2}+C\|\nabla\boldsymbol{u}\|_{L^{3}}^{2}\|
\boldsymbol{b}\|_{L^{6}}^{2}\\
\le & \frac{\nu }{4} \|\Delta\boldsymbol{b}\|_{L^{2}}^{2}+C\|\nabla\boldsymbol{u}\|_{L^{3}}^{3}
+\|\nabla\boldsymbol{b}\|_{L^{2}}^{4}+C\|\boldsymbol{b}\|_{L^{6}}^{6}\\
\le & \frac{\nu }{4} \|\Delta\boldsymbol{b}\|_{L^{2}}^{2}+C\|\nabla\boldsymbol{u}\|_{L^{2}}^{2}
\|\nabla ^{2}\boldsymbol{u}\|_{L^{2}}+C\|\nabla \boldsymbol{b}\|_{L^{2}}^{4}+C\|\boldsymbol{b}\|
_{L^{2}}^{2}\|\nabla\boldsymbol{b}\|_{L^{2}}^{4}\\
\le & \frac{\nu }{4} \|\Delta\boldsymbol{b}\|_{L^{2}}^{2}+C\|\nabla\boldsymbol{u}\|_{L^{2}}^{2}
\|\nabla ^{2}\boldsymbol{u}\|_{L^{2}}+C(1+\|\boldsymbol{b}\|_{L^{2}}^{2})\|\nabla\boldsymbol{b}\|
_{L^{2}}^{4}\\
\le & \frac{\nu }{4} \|\Delta\boldsymbol{b}\|_{L^{2}}^{2}+C\|\nabla\boldsymbol{u}\|_{L^{2}}^{2}
\|\nabla ^{2}\boldsymbol{u}\|_{L^{2}}+C\|\nabla\boldsymbol{b}\|_{L^{2}}^{4},
\end{aligned}
\end{equation}
which combined with \eqref{R12'} and \eqref{R13'} leads to
\begin{equation}\label{R15}
\begin{aligned}
&\frac{d}{dt}\left(\frac{\mu}{2}\|\nabla\boldsymbol{u}\|^{2} + \|\nabla \boldsymbol{b}\|_{L^2}^2 + \frac{\kappa}{2}\|\nabla \theta\|_{L^2}^2 +  \int \boldsymbol{b}\cdot\nabla\boldsymbol{u}\cdot
\boldsymbol{b} dx \right)\\&+ \frac{1}{2}\|\sqrt{\rho}\boldsymbol{u}_{t}\|_{L^{2}}^{2}+\frac{1}{2}\|\sqrt{\rho}\theta _{t}\|_{L^{2}}^{2}+\frac{\nu}{2}
\|\Delta\boldsymbol{b}\|_{L^{2}}^{2}\\
\le &  C(\|\sqrt{\rho}\boldsymbol{u}\|_{L^{4}}^{2}+\|\nabla\boldsymbol{u}\|_{L^{2}}) \|\nabla\boldsymbol{u}\|_{L^{2}} \|\nabla^{2}\boldsymbol{u}\|_{L^{2}}+C\|\sqrt{\rho}\boldsymbol{u}\|_{L^{4}}^{2} \|\nabla\theta\| _{L^{2}}\|\nabla^{2}\theta\| _{L^{2}}\\&
+ C\|\nabla\boldsymbol{b}\|_{L^{2}}^{4}+C\|\boldsymbol{b}\|_{L^4} \|\boldsymbol{u}\|_{H^1}^\frac{1}{2}\|\nabla \boldsymbol{b}\|_{L^{2}} \|\nabla \boldsymbol{u}\|_{L^{2}}^{\frac{3}{4}} \|\nabla^2 \boldsymbol{u}\|_{L^{2}}^{\frac{3}{4}} + C.
\end{aligned}
\end{equation}
Where
\begin{align*} M(t)=:\int(\mu|\nabla\boldsymbol{u}|^{2}+2\boldsymbol{b}\cdot\nabla
\boldsymbol{u}\cdot
\boldsymbol{b}+\kappa|\nabla\theta|^{2}+|\nabla\boldsymbol{b}|^{2})dx,
\end{align*}
satisfies
\begin{equation}
\begin{aligned}\label{R16}
&\frac{\mu}{2}\|\nabla\boldsymbol{u}\|_{L^{2}}^{2}+\|\nabla\boldsymbol{b}\|
_{L^{2}}^{2}+\kappa
\|\nabla\theta\| _{L^{2}}^{2}-C_{1}\|\boldsymbol{b}\|_{L^{4}}^{4}\\
&\le M(t)\le C\|\nabla\boldsymbol{u}\|_{L^{2}}^{2}+\|\nabla\boldsymbol{b}\|_{L^{2}}^{2}+
\|\nabla\theta\| _{L^{2}}^{2},
\end{aligned}
\end{equation}
owing to Gagliardo-Nirenberg inequlity, and following estimate:
\begin{equation}\label{R17}
|2\int\boldsymbol{b}\cdot\nabla\boldsymbol{u}\cdot\boldsymbol{b}dx|\le2
\|\nabla\boldsymbol{u}\|_{L^{2}}\|\boldsymbol{b}\|_{L^{4}}^{2}
\le \frac{\mu}{2}\|\nabla\boldsymbol{u}\|_{L^{2}}^{2}+C_{1}\|\boldsymbol{b}\|_{L^{4}}^{4},
\end{equation}
where $C_{1}$ is a positive constant.
4.Acoording to $\eqref{Benard 2}_{3}$, and multiplying it by $|\boldsymbol{b}|^{2}\boldsymbol{b}$, integrating over $\mathbb{R}^{2} $, then we can get
\begin{equation}\label{R19}
\begin{aligned}
&\quad\frac{1}{4} \frac{d}{dt}\int|\boldsymbol{b}|^{4}dx+\||\boldsymbol{b}||\nabla \boldsymbol{b}|\|
_{L^{2}}^{2}+\frac{1}{2}\|\nabla|\boldsymbol{b}|^{2}\|_{L^{2}}^{2}\\
&\le \int \boldsymbol{b}\cdot \nabla \boldsymbol{u}\cdot|\boldsymbol{b}|^{2}\boldsymbol{b}dx
\le \|\nabla\boldsymbol{u}\|_{L^{2}}\|\boldsymbol{b}\|_{L^{4}}
\||\boldsymbol{b}|^{2}\|_{L^{4}}
\\
&\le \frac{1}{4}\|\nabla|\boldsymbol{b}|^{2}\|_{L^{2}}^{2} +C\|\nabla\boldsymbol{u}\|_{L^{2}}^{2}
\|\boldsymbol{b}\|_{L^{4}}^{4} ,
\end{aligned}
\end{equation}
which implies that
\begin{equation}\label{PP}
\frac{d}{dt}\int |\boldsymbol{b}|^{4}dx+4\||\boldsymbol{b}||\nabla\boldsymbol{b}|\|_{L^{2}}^{2}
\le C\|\nabla\boldsymbol{u}\|_{L^{2}}^{2}\|\boldsymbol{b}\|_{L^{4}}^{4}, \end{equation}
then we apply Gronwall's inequality,
\begin{equation}\label{bb}
\underset{[0,T]}{\sup}\|\boldsymbol{\boldsymbol{b}}\|_{L^{4}}^{4}+\int_{0}^{T}
\||\boldsymbol{b}
||\nabla\boldsymbol{b}|\|_{L^{2}}^{2} dt\le C .
\end{equation}
Recall that $(\boldsymbol{u}, P)$ satisfies the following Stokes system
\begin{equation}\label{stokes 2}
\left\{
\begin{aligned}
&-\Delta \boldsymbol{u} + \nabla P = \rho\boldsymbol{u}_{t}-\rho\boldsymbol{u}\cdot\nabla\boldsymbol{u}
+\boldsymbol{b}
\cdot\nabla\boldsymbol{b}+\rho\theta\boldsymbol{e_{2}}, &\quad \text{in} \quad\mathbb{R}^{2}, \\
&\mathrm{div} \boldsymbol{u} =0, \quad \quad &\quad \text{in}\quad \mathbb{R}^{2}, \\
&\boldsymbol{u}(x)\to 0, \quad \quad & \quad  \left | x \right |\to +\infty .\\
\end{aligned}
\right.
\end{equation}
Applying Lemma \ref{lem 2.3}, we obtain from H\"oder's inequality and the Gagliardo-Nirenberg inequality that
\begin{equation}\label{R20}
\begin{aligned}
&\quad\|\nabla ^{2}\boldsymbol{u}\|_{L^{2}}+\|\nabla P\|_{L^{2}}\\
&\le C\|\rho\boldsymbol{u}_{t}\|_{L^{2}}+C\|\rho\boldsymbol{u}\cdot\nabla\boldsymbol{u}
\|_{L^{2}}+C\|\boldsymbol{b}\cdot\nabla\boldsymbol{b}\|_{L^{2}}+C\|\rho\theta
\boldsymbol{e_{2}}\|_{L^{2}}\\
&\le C\|\sqrt{\rho}\boldsymbol{u}_{t}\|_{L^{2}}+C\|\sqrt{\rho}\boldsymbol{u}\|_{L^{4}}
\|\nabla\boldsymbol{u}\|_{L^{4}}+C\||\boldsymbol{b}||\nabla\boldsymbol{b}||\|_{L^{2}}+
C\|\sqrt{\rho}\theta\|_{L^{2}}\\
&\le C\|\sqrt{\rho}\boldsymbol{u}_{t}\|_{L^{2}}+C\|\sqrt{\rho}\boldsymbol{u}\|_{L^{4}}
(\|\nabla\boldsymbol{u}\|_{L^{2}}^{\frac{1}{2}}\|\nabla^{2}\boldsymbol{u}\|
_{L^{2}}^{\frac{1}{2}})+C\||\boldsymbol{b}||\nabla\boldsymbol{b}||\|_{L^{2}}+C\|\sqrt{\rho}\theta
\|_{L^{2}},
\end{aligned}
\end{equation}
which implies that
\begin{equation}\label{R21}
\begin{aligned}
\|\nabla ^{2}\boldsymbol{u}\|_{L^{2}}+\|\nabla P\|_{L^{2}}&\le C\|\sqrt{\rho}\boldsymbol{u}_{t}\|_{L^{2}}+C\|\sqrt{\rho}\boldsymbol{u}
\|_{L^{4}}^{2}\|\nabla\boldsymbol{u}\|_{L^{2}}\\
&\quad+C\||\boldsymbol{b}||\nabla\boldsymbol{b}||
\|_{L^{2}}+C.
\end{aligned}
\end{equation}
Applying the classical elliptic estimates for $\theta$
\begin{equation}
\begin{aligned}
\|\nabla^{2}\theta\|_{L^{2}}&\le \|\rho\theta_{t}\|_{L^{2}}+\|\rho\boldsymbol{u}\cdot\nabla\theta
\|_{L^{2}}+\|\rho\boldsymbol{u}\cdot\boldsymbol{e_{2}}\|_{L^{2}}\\
&\le C\|\sqrt{\rho}\theta_{t}\|_{L^{2}}+C\|\sqrt{\rho}\boldsymbol{u}\|_{L^{4}}
(\|\nabla\theta \|_{L^{2}}^{\frac{1}{2}}\|\nabla^{2}\theta \|_{L^{2}}^{\frac{1}{2}})+C\|\sqrt{\rho}\boldsymbol{u}\|_{L^{2}}\\
&\le C\|\sqrt{\rho}\theta_{t}\|_{L^{2}}+C\|\sqrt{\rho}\boldsymbol{u}\|_{L^{4}}^{2}
\|\nabla\theta \|_{L^{2}}+\frac{1}{2} \|\nabla^{2}\theta \|_{L^{2}}+C,
\end{aligned}
\end{equation}
which implies that
\begin{equation}\label{R22}
\|\nabla^{2}\theta \|_{L^{2}}\le C\|\sqrt{\rho}\theta_{t}\|_{L^{2}}+C
\|\sqrt{\rho}\boldsymbol{u}\|_{L^{4}}^{2}\|\nabla\theta\|_{L^{2}}+C.
\end{equation}
Inserting \eqref{R21} and into \eqref{R15}, and combined with \eqref{bb} leads to
\begin{equation}\label{Q1}
\begin{aligned}
&\frac{d}{dt}(\int\mu|\nabla\boldsymbol{u}|^{2}+2\boldsymbol{b}\cdot\nabla
\boldsymbol{u}\cdot\boldsymbol{b}+\kappa|\nabla\theta|^{2}+|\nabla\boldsymbol{b}
|^{2})dx\\
&\quad+\|\sqrt{\rho}\boldsymbol{u}_{t}\|_{L^{2}}^{2}+\|\sqrt{\rho}\theta_{t}\|
_{L^{2}}^{2}+\nu \|\Delta\boldsymbol{b}\|_{L^{2}}^{2}
\\&\le C(\|\sqrt{\rho}
\boldsymbol{u}\|_{L^{4}}^{2}+\|\nabla\boldsymbol{u}\|_{L^{2}})\|\nabla
\boldsymbol{u}\|_{L^{2}}(\|\sqrt{\rho}\boldsymbol{u}_{t}\|_{L^{2}}+\|\sqrt
{\rho}\boldsymbol{u}\|_{L^{4}}^{2}\|\nabla\boldsymbol{u}\|_{L^{2}}
+\||\boldsymbol{b}||\nabla\boldsymbol{b}|\|_{L^{2}})\\
&\quad+C\|\sqrt{\rho}\boldsymbol{u}\|_{L^{4}}^{2}\|\nabla\theta\|(\|\sqrt{\rho}
\theta_{t}\|_{L^{2}}+C\|\sqrt
{\rho}\boldsymbol{u}\|_{L^{4}}^{2}\|\nabla\theta\|_{L^{2}})+C\|\nabla\boldsymbol
{b}\|_{L^{2}}^{4}\\
&\quad+C\|\nabla\boldsymbol{u}\|_{L^{2}}^{\frac{5}{4}}(\|\sqrt{\rho}
\boldsymbol{u}_{t}\|_{L^{2}}+\|\sqrt{\rho}\boldsymbol{u}\|_{L^{4}}^{2}\|\nabla
\boldsymbol{u}\|_{L^{2}}+\||\boldsymbol{b}||\nabla\boldsymbol{b}|\|_{L^{2}})^
{\frac{3}{4}}\|\nabla\boldsymbol{b}\|_{L^{2}}+C\\
&\le \frac{1}{2}\|\sqrt{\rho}
\boldsymbol{u}_{t}\|_{L^{2}}^{2}+C\|\sqrt{\rho}\boldsymbol{u}\|_{L^{4}}^{4}\|
\nabla\boldsymbol{u}\|_{L^{2}}^{2}+\varepsilon\||\boldsymbol{b}||\nabla
\boldsymbol{b}|\|_{L^{2}}^{2}+C\|\nabla\boldsymbol{u}\|_{L^{2}}^{4}\\
&\quad+\frac{1}{2}\|\sqrt{\rho}\theta_{t}\|_{L^{2}}^{2}+C\|\sqrt{\rho}\boldsymbol{u}
\|_{L^{4}}^{4}\|\nabla\theta\|_{L^{2}}^{2}
+C\|\nabla\boldsymbol{b}\|_{L^{2}}^{4}+C.
\end{aligned}
\end{equation}
Thus adding \eqref{bb} multiplied by $4(C_{1}+1)$ to \eqref{Q1}, choosing $\varepsilon$ suitably small, which together with \eqref{DA} in Lemma \ref{lem 2.5} gives rise to
\begin{equation}\label{Q2}
\begin{aligned}
&\frac{d}{dt}(M(t)+(C_{1}+1)\|\boldsymbol{b}\|_{L^{4}}^{4})+\|\sqrt{\rho}
\boldsymbol{u}_{t}\|_{L^{2}}^{2}+\|\sqrt{\rho}\theta_{t}\|_{L^{2}}^{2}+\nu\|
\Delta\boldsymbol{b}\|_{L^{2}}^{2}+\||\boldsymbol{b}||\nabla\boldsymbol{b}|\|
_{L^{2}}^{2}\\
&\le C\|\sqrt{\rho}\boldsymbol{u}\|_{L^{4}}^{4}(\|\nabla\boldsymbol{u}
\|_{L^{2}}^{2}+\|\nabla\theta\|_{L^{2}}^{2})+C\|\nabla\boldsymbol{u}\|_{L^{2}}^{4}+C\|\nabla\boldsymbol{b}\|_{L^{2}}^{4}+C\\
&\le C(1+\|\nabla
\boldsymbol{u}\|_{L^{2}}^{2}\log[2+\|\nabla\boldsymbol{u}\|_{L^{2}}^{2})](\|\nabla\boldsymbol{u}\|_{L^{2}}^{2}+\|\nabla\theta\|_{L^{2}}^{2})
\\&\quad+C\|\nabla\boldsymbol{u}\|_{L^{2}}^{4}+C\|\nabla\boldsymbol{b}\|_{L^{2}}^{4}+C.
\end{aligned}
\end{equation}
Put

\begin{align*}
f(t)=:2+M(t)+(C_{1}+1)\|\boldsymbol{b}\|_{L^{4}}^{4}.
\end{align*}
Then, we deduce from \eqref{Q2} that
\begin{equation}\label{Q3}
f'(t)\le C(f(t))^{2}\log f(t),
\end{equation}
Dividing \eqref{Q3} by $f$ yields
\begin{equation}\label{Q4}
(\log f(t))'\le Cf(t)\log f(t).
\end{equation}
We thus deduce from \eqref{Q4}, Gronwall's inequality, and \eqref{aa} that
\begin{equation}\label{Q6}
\begin{aligned}
\underset{[0,T]}{\sup}(&\|\nabla\boldsymbol{u}\|_{L^{2}}^{2}+\|\nabla\boldsymbol
{b}\|_{L^{2}}^{2}+\|\nabla\theta\|_{L^{2}}^{2})\le C.
\end{aligned}
\end{equation}
Integrating \eqref{Q2} with respect to $t$, we can obtain from \eqref{Q6} that
\begin{equation}\label{pp2}
\int_{0}^{T}(\|\sqrt{\rho}
\boldsymbol{u}_{t}\|_{L^{2}}^{2}+\|\sqrt{\rho}\theta_{t}\|_{L^{2}}^{2}+\|
\Delta\boldsymbol{d}\|_{L^{2}}^{2}+\||\boldsymbol{b}||\nabla\boldsymbol{b}|\|
_{L^{2}}^{2})dt\le C
\end{equation}
Multiplying \eqref{Q2} by $t$, performing similar technique to \eqref{Q3}, one can also obtain
\begin{equation}\label{Q7}
\begin{aligned}
\underset{[0,T]}{\sup}~ t(&\|\nabla\boldsymbol{u}\|_{L^{2}}^{2}+\|\nabla\boldsymbol
{b}\|_{L^{2}}^{2}+\|\nabla\theta\|_{L^{2}}^{2})\\
&+\int_{0}^{T}t(\|\sqrt{\rho}
\boldsymbol{u}_{t}\|_{L^{2}}^{2}+\|\sqrt{\rho}\theta_{t}\|_{L^{2}}^{2}+\|
\Delta\boldsymbol{d}\|_{L^{2}}^{2}+\||\boldsymbol{b}||\nabla\boldsymbol{b}|\|
_{L^{2}}^{2})dt\le C.
\end{aligned}
\end{equation}
The proof of Lemma \ref{lem 3.2} is completed.
\end{proof}

\begin{remark} Based on Lemma \ref{lem 3.2}, we obtain the following useful estimates. First, we deduce from \eqref{DA} that
\begin{equation}\label{Q8}
\|\sqrt{\rho}\boldsymbol{u}\|_{L^{4}}^{2}\le C(\bar{\rho})(1+\|\sqrt{\rho}
\boldsymbol{u}\|_{L^{2}})\|\boldsymbol{u}\|_{H^{1}}\sqrt{\log(2+\|\boldsymbol{u}
\|_{H^{1}}^{2})}\le C,
\end{equation}
and
\begin{equation}\label{w1}
\begin{aligned}
\||\boldsymbol{b}||\nabla\boldsymbol{b}|\|_{L^{2}}^{2}&\le \|\boldsymbol{b}\|_{L^{4}}^{2}
\|\nabla\boldsymbol{b}\|_{L^{4}}^{2}\le C\|\nabla\boldsymbol{b}\|_{L^{2}}^{2}+C\|\nabla^{2}
\|\boldsymbol{b}\|_{L^{4}}^{2}\\
&\le C\|\boldsymbol{b}_{t}\|_{L^{2}}^{2}+C\|\boldsymbol{u}\cdot\nabla\boldsymbol{b}\|_{L^{2}}^{2}
+C\|\boldsymbol{b}\cdot\nabla\boldsymbol{u}\|_{L^{2}}^{2}\\
&\le C\|\boldsymbol{b}_{t}\|_{L^{2}}^{2}+C\|\boldsymbol{u}\|_{L^{\infty}}^{2}\|\nabla\boldsymbol{b}\|
_{L^{2}}^{2}+C\|\boldsymbol{b}\|_{L^{4}}^{2}\|\nabla\boldsymbol{u}\|_{L^{4}}^{2}\\
&\le C\|\boldsymbol{b}_{t}\|_{L^{2}}^{2}+C\|\boldsymbol{u}\|_{L^{4}}\|\nabla\boldsymbol{u}\|_{L^{4}}
+C\|\nabla\boldsymbol{u}\|_{L^{2}}\|\nabla\boldsymbol{u}\|_{L^{2}}\\
&\le C\|\boldsymbol{b}_{t}\|_{L^{2}}+\frac{1}{2}\|\nabla^{2}\boldsymbol{u}\|_{L^{2}}^{2}+C.
\end{aligned}
\end{equation}
This together with \eqref{R21} and \eqref{R22} gives
\begin{equation}\label{Q9}
\begin{aligned}
\|\nabla^{2}\boldsymbol{u}\|_{L^{2}}^{2}&\le C\|\sqrt{\rho}\boldsymbol{u}_{t}\|
_{L^{2}}^{2}+C\|\sqrt{\rho}\boldsymbol{u}\|_{L^{4}}^{2}\|\nabla\boldsymbol{u}\|
_{L^{2}} + C\||\boldsymbol{b}||\nabla\boldsymbol{b}|\|_{L^{2}}^{2}+C\\&\le
C\|\sqrt{\rho}\boldsymbol{u}_{t}\|_{L^{2}}^{2}+C\||\boldsymbol{b}||\nabla\boldsymbol
{b}|\|_{L^{2}}^{2} +C\\
&\le C\|\sqrt{\rho}\boldsymbol{u}_{t}\|_{L^{2}}^{2}+C\|\boldsymbol{b}_{t}\|_{L^{2}}^{2}+
\frac{1}{2}\|\nabla^{2}\boldsymbol{u}\|_{L^{2}}^{2}+C,
\end{aligned}
\end{equation}
which implies
\begin{equation}\label{q1}
\begin{aligned}
\|\nabla^{2}\boldsymbol{u}\|_{L^{2}}^{2}
&\le C\|\sqrt{\rho}\boldsymbol{u}_{t}\|_{L^{2}}^{2}
+C\|\boldsymbol{b}_{t}\|_{L^{2}}^{2}+C.
\end{aligned}
\end{equation}
Similarily, one has
\begin{equation}\label{Q10}
\begin{aligned}
\|\nabla^{2}\theta\|_{L^{2}}^{2}&\le C\|\sqrt{\rho}\theta _{t}\|
_{L^{2}}^{2}+C\|\sqrt{\rho}\boldsymbol{u}\|_{L^{4}}^{2}\|\nabla\theta \|
_{L^{2}}+C\\&\le C\|\sqrt{\rho}\theta _{t}\|_{L^{2}}^{2} + C.
\end{aligned}
\end{equation}
\end{remark}
Due to lack of the Choe-Kim type compatibility condition on the initial data, we derive the time-weighted energy estimate of $(\sqrt{\rho}\boldsymbol{u}_t, \sqrt{\rho}\theta_t, \boldsymbol{b}_t).$
\begin{lemma}\label{lem 3.3}
It holds that for all $0\le T<T^{*}$,
\begin{equation}\label{Q11}
\begin{aligned}
\sup_{[0, T]} t(\|\sqrt{\rho}\boldsymbol{u}_{t}\|_{L^{2}}^{2}+ & \|\sqrt
{\rho}\theta _{t}\|_{L^{2}}^{2}+\|\boldsymbol{b}_{t}\|_{L^{2}}^{2}) \\
 & +\int_{0}^{T}
t(\|\nabla \boldsymbol{u}_{t}\|_{L^{2}}^{2}+\|\nabla\theta_{t}\|_{L^{2}}^{2}+\|
\nabla\boldsymbol{b}_{t}\|_{L^{2}}^{2})\le C.
\end{aligned}
\end{equation}
\end{lemma}
\begin{proof}
First, differentiating $\eqref{Benard 2}_{2}$ with respect to $t$, we have
\begin{equation}\label{Q12}
\begin{aligned}
&\quad\rho\boldsymbol{u}_{tt}+\rho\boldsymbol{u}\cdot\nabla\boldsymbol{u}_{t}-\Delta \boldsymbol{u}_{t} + \nabla P_{t} \\
& = -\rho_{t} \boldsymbol{u}
_{t}-\rho_t \boldsymbol{u} \cdot\nabla\boldsymbol{u} -\rho \boldsymbol{u}_t \cdot\nabla\boldsymbol{u} +(\rho\theta)_{t}\boldsymbol
{e}_{2}+\boldsymbol{b}_{t}\cdot\nabla\boldsymbol{b}+\boldsymbol{b}\cdot\nabla
\boldsymbol{b}_{t}.
\end{aligned}
\end{equation}
Multiplying \eqref{Q12} by $\boldsymbol{u}_{t}$ and integrating by parts over $\mathbb{R}^{2}$ give
\begin{equation}\label{Q13}
\begin{aligned}
&\frac{1}{2}\frac{d}{dt}\int\rho|\boldsymbol{u}_{t}|^{2}dx+\mu\int|\nabla\boldsymbol
{u}_{t}|^{2}dx\\
= & -\int\rho_{t}|\boldsymbol{u}_{t}|^{2}dx- \int \rho_t \boldsymbol{u}
\cdot\nabla\boldsymbol{u}\cdot\boldsymbol{u}_{t}dx - \int \rho \boldsymbol{u}_t
\cdot\nabla\boldsymbol{u}\cdot\boldsymbol{u}_{t}dx\\
&+\int(\rho\theta)_{t}(\boldsymbol
{e}_{2}\cdot\boldsymbol{u}_{t})dx+\int\boldsymbol{b}_{t}\cdot\nabla\boldsymbol{b}
\cdot\boldsymbol{u}_{t}dx+\int\boldsymbol{b}\cdot\nabla\boldsymbol{b}_{t}\cdot
\boldsymbol{u}_{t}dx.
\end{aligned}
\end{equation}
Next, differentiating $\eqref{Benard 2}_{4}$ with respect to $t$, we get
\begin{equation}\label{Q14}
\rho\theta_{tt}+\rho\boldsymbol{u}\cdot\nabla \theta_{t}-\kappa\Delta \theta
_{t} = -\rho_{t}\theta_{t}-\rho_t \boldsymbol{u} \cdot\nabla\theta -\rho \boldsymbol{u}_t\cdot\nabla\theta+
(\rho\boldsymbol{u}\cdot\boldsymbol{e}_{2})_{t}.
\end{equation}
Multiplying $\eqref{Q14}$ by $\theta_{t}$ and integrating over $\mathbb{R}^{2}$ yield
\begin{equation}\label{Q15}
\begin{aligned}
&\quad\frac{1}{2}\frac{d}{dt}\int\rho \theta_{t}^{2}dx +\kappa\int|\nabla\theta_{t}|^{2}dx\\
&=-\int\rho_{t} \theta_{t}^{2}dx-\int \rho_t (\boldsymbol{u} \cdot\nabla\theta)
\theta_{t}dx -\int \rho (\boldsymbol{u}_t \cdot\nabla\theta) \theta_{t}dx +\int(\rho\boldsymbol{u} \boldsymbol{e}_{2})_{t}\cdot\theta
_{t}dx.
\end{aligned}
\end{equation}
Finally, differentiating $\eqref{Benard 2}_{3}$ with respect to $t$ leads to
\begin{equation}\label{Q16}
\boldsymbol{b}_{tt} + \boldsymbol{u}\cdot\nabla\boldsymbol{b}_{t} -\nu \Delta \boldsymbol{b}_{t} = - \boldsymbol{u}_{t}
\cdot\nabla\boldsymbol{b} -\boldsymbol
{b}_{t}\cdot\nabla\boldsymbol{u}-\boldsymbol{b}\cdot\nabla \boldsymbol{u}_{t}.
\end{equation}
Multiplying \eqref{Q16} by $\boldsymbol{b}_{t}$ and integrating over $\mathbb{R}^{2}$, we obtain
\begin{equation}\label{Q17}
\begin{aligned}
&\quad\frac{1}{2}\frac{d}{dt}\int|\boldsymbol{b}_{t}|^{2}dx +\nu\int|\nabla\boldsymbol{b}_{t}|^{2}
dx\\
&=-\int\boldsymbol{u}_{t}\cdot\nabla\boldsymbol{b}\cdot\boldsymbol{b}_{t}dx-
\int\boldsymbol{b}_{t}\cdot\nabla
\boldsymbol{u}\cdot\boldsymbol{b}_{t}dx+\int\boldsymbol{b}\cdot\nabla
\boldsymbol{u}_{t}\cdot\boldsymbol{b}_{t}dx.
\end{aligned}
\end{equation}
Adding \eqref{Q15}, \eqref{Q17} and \eqref{Q13}, we deduce from the mass equation $\eqref{Benard}_1$ that
\begin{equation}\label{Q18}
\begin{aligned}
& \frac{1}{2}\frac{d}{dt}\int \left(\rho|\boldsymbol{u}_{t}|^{2}+\rho \theta _{t}^{2}+
|\boldsymbol{b}_{t}|^{2}\right)dx + \int \left(\mu|\nabla \boldsymbol{u}_{t}|^{2}+\kappa
|\nabla\theta_{t}|^{2}+\nu|\nabla\boldsymbol{b}_{t}|^{2}\right) dx \\
= & \int \mathrm{div}(\rho\boldsymbol{u}) |\boldsymbol{u}_{t}|^{2} dx + \int \mathrm{div}(\rho\boldsymbol{u}) \theta_{t}^{2} dx  +  \int \mathrm{div}(\rho\boldsymbol{u})(\boldsymbol{u} \cdot\nabla\boldsymbol{u}\cdot\boldsymbol{u}_{t})dx  \\
& + \int\mathrm{div}(\rho
\boldsymbol{u})(\boldsymbol{u}\cdot\nabla\theta\cdot\theta_{t})
dx - \int\rho\boldsymbol{u}_{t}\cdot\nabla\boldsymbol{u}\cdot\boldsymbol{u}_{t}dx
-\int\rho (\boldsymbol{u}_{t}\cdot\nabla\theta)\theta_{t}dx \\
& - \int\boldsymbol{u}_{t}\cdot\nabla\boldsymbol{b}\cdot\boldsymbol{b}_{t}dx
+\int\boldsymbol{b}_{t}\cdot\nabla\boldsymbol{b}\cdot\boldsymbol{u}_{t}dx
+\int\boldsymbol{b}_{t}\cdot\nabla\boldsymbol{u}\cdot\boldsymbol{b}_{t}dx \\
& + 2\int\rho\theta_{t}(\boldsymbol{e}_{2}\cdot\boldsymbol{u}_{t})dx
+\int \mathrm{div}(\rho\boldsymbol{u}) (\theta(\boldsymbol{e}_{2}\cdot\boldsymbol{u}_{t}))dx + \int \mathrm{div}(\rho\boldsymbol{u}) (\theta_t (\boldsymbol{e}_{2}\cdot\boldsymbol{u}))dx
\\
\triangleq & \sum_{i=1}^{12}J_{i}.
\end{aligned}
\end{equation}
Before proceeding further, we introduce the following useful conclusions. It follows from \eqref{R3} and \eqref{z2} that
$$ \|\boldsymbol{u}_t\|_{L^2}^2 \le C(\tilde{\rho}, \|\rho_0-\tilde{\rho}\|_{L^2}^2)(\|\sqrt{\rho}\boldsymbol{u}_t\|_{L^2}^2 + \|\nabla \boldsymbol{u}_t\|_{L^2}^2),$$
and
$$ \|\theta_t\|_{L^2}^2 \le C(\tilde{\rho}, \|\rho_0-\tilde{\rho}\|_{L^2}^2)(\|\sqrt{\rho}\theta_t\|_{L^2}^2 + \|\nabla \theta_t\|_{L^2}^2).$$
This implies that
\begin{equation} \label{ll01}
\|\boldsymbol{u}_t\|_{H^1}^2 \le C (\|\sqrt{\rho}\boldsymbol{u}_t\|_{L^2}^2 + \|\nabla \boldsymbol{u}_t\|_{L^2}^2),
\end{equation}
and
\begin{equation} \label{ll02}
\|\theta_t\|_{H^1}^2 \le C (\|\sqrt{\rho}\theta_t\|_{L^2}^2 + \|\nabla \theta_t\|_{L^2}^2).
\end{equation}
Now, we are ready to estimate the terms of the right-hand side of \eqref{Q18}. By H\"older's inequality, the Gagliardo-Nirenberg inequality, and \eqref{R11}, we get
\begin{align}
\nonumber
\left|J_1\right| & = \left| -\int \rho \boldsymbol{u} \cdot \nabla|\boldsymbol{u}_t|^2 d x \right| \\ \nonumber
& \leq 2\|\sqrt{\rho} \boldsymbol{u}\|_{L^4} \|\sqrt{\rho} \boldsymbol{u}_t\|_{L^4} \|\nabla \boldsymbol{u}_t \|_{L^2} \\ \nonumber
& \leq \frac{\mu}{20} \|\nabla \boldsymbol{u}_t \|_{L^2}^2 + C \|\sqrt{\rho} \boldsymbol{u}_t \|_{L^2}^{\frac{1}{2}} \|\sqrt{\rho}\boldsymbol{u}_t \|_{L^6}^{\frac{3}{2}} \\ \nonumber
& \leq \frac{\mu}{20}\left\|\nabla \boldsymbol{u}_t\right\|_{L^2}^2 + C\|\rho\|_{L^{\infty}}^{\frac{3}{4}}\left\|\sqrt{\rho} \boldsymbol{u}_t\right\|_{L^2}^{\frac{1}{2}}\left\|\boldsymbol{u}_t\right\|_{H^1}^{\frac{3}{2}} \\ \nonumber
& \leq \frac{\mu}{10}\left\|\nabla \boldsymbol{u}_t\right\|_{L^2}^2+C\left\|\sqrt{\rho} \boldsymbol{u}_t\right\|_{L^2}^2, \\
\left|J_2 \right|& = \left| -\int \rho \boldsymbol{u} \cdot \nabla |\theta_t|^2 d x \right| \\ \nonumber
& \leq 2\|\sqrt{\rho} \boldsymbol{u}\|_{L^4} \|\sqrt{\rho} \theta_t\|_{L^4} \|\nabla \theta_t \|_{L^2} \\ \nonumber
& \leq \frac{\kappa}{12} \|\nabla \theta_t \|_{L^2}^2 + C \|\sqrt{\rho} \theta_t \|_{L^2}^{\frac{1}{2}} \|\sqrt{\rho}\theta_t \|_{L^6}^{\frac{3}{2}} \\ \nonumber
& \leq \frac{\kappa}{12}\left\|\nabla \theta_t\right\|_{L^2}^2 + C\|\rho\|_{L^{\infty}}^{\frac{3}{4}}\left\|\sqrt{\rho} \theta_t\right\|_{L^2}^{\frac{1}{2}}\left\|\theta_t\right\|_{H^1}^{\frac{3}{2}} \\ \nonumber
& \leq \frac{\kappa}{6}\left\|\nabla \theta_t\right\|_{L^2}^2+C\left\|\sqrt{\rho} \theta_t\right\|_{L^2}^2, \\
\left|J_{3}\right|  & =  \left|-\int \rho\boldsymbol{u}\cdot \nabla(\boldsymbol{u}\cdot\nabla\boldsymbol{u}\cdot
\boldsymbol{u}_{t})dx\right| \\ \nonumber
& \le \int\rho|\boldsymbol{u}||\nabla\boldsymbol{u}|^{2}|\boldsymbol{u}_{t}|dx
+\int\rho|\boldsymbol{u}|^{2}|\nabla^{2}\boldsymbol{u}||\boldsymbol{u}_{t}|dx +\int\rho|\boldsymbol
{u}|^{2}|\nabla\boldsymbol{u}||\nabla\boldsymbol{u}_{t}|dx \\ \nonumber
& \le  \|\rho\|_{L^{\infty}}^{\frac{1}{2}} \|\boldsymbol{u}\|_{L^{\infty}}\|\nabla
\boldsymbol{u}\|_{L^{4}}^2 \|\sqrt{\rho}\boldsymbol{u}_{t}\|_{L^{2}} +\|\rho\|_{L^{\infty}}^{\frac{1}{2}} \|\boldsymbol{u}\|_{L^{\infty}}^{2}\|\nabla^{2}\boldsymbol{u}\|_{L^{2}}\|\sqrt{\rho}\boldsymbol{u}_{t}\|_{L^{2}}\\ \nonumber
& \quad +\|\rho\|_{L^{\infty}} \|\boldsymbol{u}\|_{L^{\infty}}^{2}\|\nabla\boldsymbol{u}\|_{L^{2}}\|\nabla\boldsymbol{u}_{t}\|_{L^{2}}\\ \nonumber
& \le C\|\boldsymbol{u}\|_{L^2}^{\frac{1}{3}} \|\nabla \boldsymbol{u}\|_{L^4}^{\frac{8}{3}} \left\|\sqrt{\rho} \boldsymbol{u}_t\right\|_{L^2}+C\|\boldsymbol{u}\|_{L^2}^{\frac{2}{3}}\|\nabla \boldsymbol{u}\|_{L^4}^{\frac{4}{3}}\left\|\nabla^2 \boldsymbol{u}\right\|_{L^2}\left\|\sqrt{\rho} \boldsymbol{u}_t\right\|_{L^2} \\ \nonumber
& \quad +C\|\boldsymbol{u}\|_{L^2}^{\frac{2}{3}}\|\nabla \boldsymbol{u}\|_{L^4}^{\frac{4}{3}}\|\nabla \boldsymbol{u}\|_{L^2}\left\|\nabla \boldsymbol{u}_t\right\|_{L^2} \\ \nonumber
& \le C\left\|\sqrt{\rho} \boldsymbol{u}_t\right\|_{L^2}^2+C\|\nabla \boldsymbol{u}\|_{L^2}^{\frac{8}{3}}\left\|\nabla^2 \boldsymbol{u}\right\|_{L^2}^{\frac{8}{3}}+C\|\nabla \boldsymbol{u}\|_{L^4}^{\frac{8}{3}}\left\|\nabla^2 \boldsymbol{u}\right\|_{L^2}^2 \\ \nonumber
& \quad +C\|\nabla \boldsymbol{u}\|_{L^2}^{\frac{5}{3}}\left\|\nabla^2 \boldsymbol{u}\right\|_{L^2}^{\frac{2}{3}}\left\|\nabla \boldsymbol{u}_t\right\|_{L^2} \\ \nonumber
& \leq  \frac{\mu}{10}\left\|\nabla \boldsymbol{u}_t\right\|_{L^2}^2+C\left\|\sqrt{\rho} \boldsymbol{u}_t\right\|_{L^2}^2+C\|\nabla \boldsymbol{u}\|_{L^2}^{\frac{8}{3}}\left(\left\|\sqrt{\rho} \boldsymbol{u}_t\right\|_{L^2}+\|\nabla \boldsymbol{u}\|_{L^2}+\||\boldsymbol{b}||\nabla \boldsymbol{b}| \|_{L^2}\right)^{\frac{8}{3}} \\ \nonumber
& \quad  +C\|\nabla \boldsymbol{u}\|_{L^2}^{\frac{4}{3}}\left(\left\|\sqrt{\rho} \boldsymbol{u}_t\right\|_{L^2}+\|\nabla \boldsymbol{u}\|_{L^2}+ \||\boldsymbol{b}||\nabla \boldsymbol{b}| \|_{L^2}\right)^{\frac{10}{3}} \\ \nonumber
& \quad +C\|\nabla \boldsymbol{u}\|_{L^2}^{\frac{10}{3}}\left(\left\|\sqrt{\rho} \boldsymbol{u}_t\right\|_{L^2}+\|\nabla \boldsymbol{u}\|_{L^2}+\||\boldsymbol{b}||\nabla \boldsymbol{b}|\|_{L^2}\right)^{\frac{4}{3}} \\ \nonumber
& \leq \frac{\mu}{10}+C\left\|\sqrt{\rho} \boldsymbol{u}_t\right\|_{L^2}^4+C\|\boldsymbol{b}_{t}\|_{L^2}^4+C,\\ \nonumber
J_{4}& =\left|-\int \rho\boldsymbol{u}\cdot \nabla(\boldsymbol{u}\cdot\nabla\theta\cdot
\theta_{t})dx\right|\\
&\le \int\rho|\boldsymbol{u}||\nabla\boldsymbol{u}||\nabla\theta||\theta_{t}|dx
+\int\rho|\boldsymbol{u}|^{2}|\nabla^{2}\theta||\theta_{t}|dx \\
&\quad +\int\rho|\boldsymbol{u}|^{2}|\nabla\theta||\nabla\theta_{t}|dx \\
&\le  \|\rho\|_{L^{\infty}}^{\frac{1}{2}} \|\boldsymbol{u}\|_{L^{\infty}}\|\nabla
\boldsymbol{u}\|_{L^{4}} \|\nabla\theta\|_{L^{4}} \|\sqrt{\rho}\theta_{t}\|_{L^{2}} \\
& \quad +\|\rho\|_{L^{\infty}}^{\frac{1}{2}} \|\boldsymbol{u}\|_{L^{\infty}}^{2}\|\nabla^{2}
\theta\|_{L^{2}}\|\sqrt{\rho}\theta_{t}\|_{L^{2}}\\
& \quad +\|\rho\|_{L^{\infty}} \|\boldsymbol{u}\|_{L^{\infty}}^{2}\|\nabla\theta\|_{L^{2}}
\|\nabla\theta_{t}\|_{L^{2}}\\
&\le C\|\boldsymbol{u}\|_{L^{2}}^{\frac{1}{3}}\|\nabla\boldsymbol{u}\|_{L^{4}}^{\frac{5}{3}}
\|\sqrt{\rho}\theta_{t}\|_{L^{2}}\|\nabla\theta\|_{L^{2}}^{\frac{1}{2}}
\|\nabla^{2}\theta\|_{L^{2}}^{\frac{1}{2}}\\
&\quad +C\|\boldsymbol{u}\|_{L^{2}}^{\frac{2}{3}}\|\nabla\boldsymbol{u}\|_{L^{4}}^{\frac{4}{3}}
\|\nabla^{2}\theta\|_{L^{2}}\|\sqrt{\rho}\theta_{t}\|_{L^{2}}\\
&\quad +C\|\boldsymbol{u}\|_{L^{2}}^{\frac{2}{3}}\|\nabla\boldsymbol{u}\|_{L^{4}}^{\frac{4}{3}}
\|\nabla\theta\|_{L^{2}}\|\nabla\theta_{t}\|_{L^{2}}\\
&\le C\|\sqrt{\rho}\theta_{t}\|_{L^{2}}(1+\|\nabla\boldsymbol{u}\|_{L^{2}})^{\frac{1}{3}}\|\nabla
\boldsymbol{u}\|_{L^{2}}^{\frac{5}{3}}\|\nabla^{2}\boldsymbol{u}\|^{\frac{5}{6}}
\|\nabla\theta\|_{L^{2}}^{\frac{1}{2}}\|\nabla^{2}\theta\|_{L^{2}}^{\frac{1}{2}}\\
&\quad +C(1+\|\nabla\boldsymbol{u}\|)^{\frac{2}{3}}\|\nabla^{2}
\boldsymbol{u}\|_{L^{2}}^{\frac{2}{3}}\|\nabla
\boldsymbol{u}\|_{L^{2}}^{\frac{2}{3}}\|\nabla^{2}\theta\|_{L^{2}}\|\sqrt{\rho}\theta_{t}
\|_{L^{2}}\\
&\quad +C(1+\|\nabla\boldsymbol{u}\|)^{\frac{2}{3}}
\|\nabla\boldsymbol{u}\|_{L^{2}}^{\frac{4}{3}}\|\nabla\theta\|_{L^{2}}
\|\nabla\theta_{t}\|_{L^{2}}\\
&\le \frac{\kappa}{6}\|\nabla \theta_{t}\|_{L^{2}}^{2}+C\|\sqrt{\rho}\theta_{t}\|_{L^{2}}\|\nabla^{2}\boldsymbol{u}\|^{\frac{5}{6}}
\|\nabla^{2}\theta\|_{L^{2}}^{\frac{1}{2}}\\&\quad+C\|\sqrt{\rho}\boldsymbol{u}\|_{L^{2}}
\|\nabla^{2}\boldsymbol{u}\|^{\frac{2}{3}}\|\nabla^{2}\theta\|_{L^{2}}+C\\
&\le  \frac{\kappa}{6}\|\nabla \theta_{t}\|_{L^{2}}^{4}+
C\|\sqrt{\rho}\theta_{t}\|_{L^{2}}^{4}+C\|\sqrt{\rho}\boldsymbol{u}\|_{L^{2}}^{4}
+\|\boldsymbol{b}_{t}\|_{L^{2}}^{2}+C,\\
\left|J_{5}\right| & \le  \|\nabla \boldsymbol{u}\|_{L^2} \|\sqrt{\rho} \boldsymbol{u}_{t}\|_{L^4}^2 \\ \nonumber
& \le C \|\sqrt{\rho} \boldsymbol{u}_{t}\|_{L^2}^{\frac{1}{2}}\|\sqrt{\rho} \boldsymbol{u}_{t}\|_{L^6}^{\frac{3}{2}}   \\ \nonumber
& \le  C\|\rho\|_{L^{\infty}}^{\frac{3}{4}}\|\sqrt{\rho}\boldsymbol{u}_{t}\|_{L^{2}}^{\frac{1}{2}} \|\nabla \boldsymbol{u}_{t}\|_{H^1}^{\frac{3}{2}} \\ \nonumber
& \le  \frac{\mu}{10} \|\nabla\boldsymbol{u}_{t}\|_{L^{2}}^{2}
+C\|\sqrt{\rho}\boldsymbol{u}_{t}\|_{L^{2}}^{2}, \\ \nonumber
\left|J_{6}\right| & \le  \|\nabla \theta\|_{L^2} \|\sqrt{\rho} \boldsymbol{u}_{t}\|_{L^4} \|\sqrt{\rho} \theta_{t}\|_{L^4}\\ \nonumber
& \le C \|\sqrt{\rho} \boldsymbol{u}_{t}\|_{L^4}^2 + C \|\sqrt{\rho} \theta_{t}\|_{L^4}^{2}\\ \nonumber
& \le C \|\sqrt{\rho} \boldsymbol{u}_{t}\|_{L^2}^{\frac{1}{2}}\|\sqrt{\rho} \boldsymbol{u}_{t}\|_{L^6}^{\frac{3}{2}} + C \|\sqrt{\rho} \theta_{t}\|_{L^2}^{\frac{1}{2}}\|\sqrt{\rho} \theta_{t}\|_{L^6}^{\frac{3}{2}}  \\ \nonumber
& \le  C\|\rho\|_{L^{\infty}}^{\frac{3}{4}}\|\sqrt{\rho}\boldsymbol{u}_{t}\|_{L^{2}}^{\frac{1}{2}} \|\nabla \boldsymbol{u}_{t}\|_{H^1}^{\frac{3}{2}}  + C\|\rho\|_{L^{\infty}}^{\frac{3}{4}}\|\sqrt{\rho}\theta_{t}\|_{L^{2}}^{\frac{1}{2}} \|\nabla \theta_{t}\|_{H^1}^{\frac{3}{2}}\\ \nonumber
& \le  \frac{\mu}{10} \|\nabla\boldsymbol{u}_{t}\|_{L^{2}}^{2}
+  \frac{\kappa}{6} \|\nabla\theta_{t}\|_{L^{2}}^{2}
+C\|\sqrt{\rho}\boldsymbol{u}_{t}\|_{L^{2}}^{2}
+C\|\sqrt{\rho}\theta_{t}\|_{L^{2}}^{2}, \\ \nonumber
\left|J_{7} + J_{8} \right|
  & \le C \|\boldsymbol{u}_t\|_{L^4} \|\nabla \boldsymbol{b}\|_{L^2} \|\boldsymbol{b}_t\|_{L^4} \\ \nonumber
& \le C(\|\sqrt{\rho}\boldsymbol{u}_t\|_{L^2} + \|\nabla \boldsymbol{u}_t\|_{L^2}) \|\boldsymbol{b}_t\|_{L^4} \\ \nonumber
& \le  \frac{\mu}{10}\|\nabla \boldsymbol{u}_t\|_{L^{2}}^2 + C \|\sqrt{\rho}\boldsymbol{u}_t\|_{L^{2}}^2 + C \|\boldsymbol{b}_t\|_{L^{2}}
\|\nabla\boldsymbol{b}_{t}\|_{L^{2}}\\ \nonumber
& \le  \frac{\mu}{10}\|\nabla \boldsymbol{u}_t\|_{L^{2}}^2 + \frac{\kappa}{10}\|\nabla \theta_t\|_{L^{2}}^2 +  C \|\sqrt{\rho}\boldsymbol{u}_t\|_{L^{2}}^2 + C \|\boldsymbol{b}_t\|_{L^{2}}^2,\\ \nonumber
\left|J_{9} \right|
& \le C \|\boldsymbol{b}_t\|_{L^4}^2 \|\nabla \boldsymbol{u}\|_{L^2}  \\ \nonumber
& \le C \|\boldsymbol{b}_t\|_{L^2} \|\nabla \boldsymbol{b}_t\|_{L^2}\\ \nonumber
& \le  \frac{1}{4}\|\nabla\boldsymbol{b}_{t}\|_{L^{2}}^2 + C\|\boldsymbol{b}_t\|_{L^{2}}^2, \\ \nonumber
\left|J_{10} \right| & \le \left|2\int\rho\theta_{t}(\boldsymbol{e}_{2}\cdot\boldsymbol{u}_{t})dx \right| \\ \nonumber
& \le C\|\sqrt{\rho}\boldsymbol{u}_{t}\|_{L^{2}}^{2}+C\|\sqrt{\rho}\theta _{t}\|_{L^{2}}^{2}, \\ \nonumber
\left|J_{11}+J_{12}\right| & \le \left|\int\rho\boldsymbol{u}\cdot\nabla (\theta(\boldsymbol{e}_{2}\cdot
\boldsymbol{u}_{t}))dx+\int \rho\boldsymbol{u}\cdot \nabla(\theta_{t}(\boldsymbol{e}_{2}\cdot\boldsymbol{u}))dx \right|\\ \nonumber
& \le C\int\rho|\boldsymbol{u}|\left(|\boldsymbol{u}_{t}||\nabla\theta|+|\theta|
|\nabla\boldsymbol{u}_{t}|+|\boldsymbol{u}||\nabla\theta_{t}|+|\theta_{t}|
|\nabla\boldsymbol{u}|\right) dx \\ \nonumber
& \leq  C\|\sqrt{\rho} \boldsymbol{u}\|_{L^{4}}\left\|\sqrt{\rho} \boldsymbol{u}_{t}\right\|_{L^{4}}\left\|\nabla \theta \right\|_{L^{2}} + C\|\sqrt{\rho} \boldsymbol{u}\|_{L^{4}}\left\|\sqrt{\rho} \theta \right\|_{L^{4}}\left\|\nabla \boldsymbol{u}_t\right\|_{L^{2}} \\ \nonumber
& \quad + C\|\sqrt{\rho} \boldsymbol{u}\|_{L^{4}}^2 \left\|\nabla \theta_t\right\|_{L^{2}} + C\|\sqrt{\rho} \boldsymbol{u}\|_{L^{4}}\left\|\sqrt{\rho} \theta_t\right\|_{L^{4}}\left\|\nabla \boldsymbol{u}\right\|_{L^{2}} \\ \nonumber
& \leq  \frac{\mu}{10}\left\|\nabla \boldsymbol{u}_{t}\right\|_{L^{2}}^{2}+ \frac{\kappa}{6}\left\|\nabla \theta_{t}\right\|_{L^{2}}^{2}  +
C\left\|\sqrt{\rho} \boldsymbol{u}_{t}\right\|_{L^{2}}^2 + C\left\|\sqrt{\rho} \theta_{t}\right\|_{L^{2}}^2 + C. \\ \nonumber
\end{align}
Substituting the above estimates into \eqref{Q18}, we obtain that
\begin{equation}\label{Q19}
\begin{aligned}
&\frac{d}{dt}\int \left(\rho|\boldsymbol{u}_{t}|^{2}+\rho \theta _{t}^{2}+
|\boldsymbol{b}_{t}|^{2} \right)dx +\int \left(\mu|\nabla \boldsymbol{u}_{t}|^{2}+\kappa
|\nabla\theta_{t}|^{2}+\nu|\nabla\boldsymbol{b}_{t}|^{2}\right) dx\\
\le &C(\|\sqrt{\rho}\theta
_{t}\|_{L^{2}}^{2}+C\|\sqrt{\rho}\boldsymbol{u}_{t}\|_{L^{2}}^{2}+
C\|\boldsymbol{b}_{t}\|_{L^{2}}^{2})^{2}+C.
\end{aligned}
\end{equation}
Multiplying \eqref{Q19} by $t$ gives
\begin{equation} \label{lll03}
\begin{aligned}
&\frac{d}{dt}\left(t\int(\rho|\boldsymbol{u}_{t}|^{2}+\rho|\theta _{t}|^{2}+
|\boldsymbol{b}_{t}|^{2})dx\right) + t( \mu\|\nabla \boldsymbol{u}_{t}\|_{L^2}^{2}+\kappa
||\nabla\theta_{t}\|_{L^2}^{2}+\nu\|\nabla\boldsymbol{b}_{t}\|_{L^2}^{2})\\
\le & Ct(\|\sqrt{\rho}\boldsymbol{u}_{t}\|_{L^{2}}^{2}+ \|\boldsymbol{b}_{t}\|
_{L^{2}}^{2}+\|\sqrt{\rho}\theta _{t}\|_{L^{2}}^{2})^{2}+ \|\sqrt{\rho}\boldsymbol{u}_{t}\|_{L^{2}}^{2}+ \|\boldsymbol{b}_{t}\|
_{L^{2}}^{2}+ \|\sqrt{\rho}\theta _{t}\|_{L^{2}}^{2} + C.
\end{aligned}
\end{equation}
Applying Gronwall's inequality yields the desired \eqref{Q11}.
\end{proof}

\begin{lemma}\label{lemma 3.4}
Let $q$ be as in $\eqref{main}$, it holds that for all $0\le T< T^{*}$ and all $2\le r<q$,
\begin{equation}
\begin{aligned}
\sup_{[0, T]} \Big(\|\rho-\tilde{\rho}\|_{H^{1}\cap W^{1,4}} &+ \|\rho_{t}\|_{L^{r}}\Big)+\int_{0}^{T}(t\|\nabla\boldsymbol{u}\|_{H^{2}}^{2}+t\|\nabla\theta\|_
{H^{2}}^{2}+t\|\nabla\boldsymbol{b}\|_{H^{2}}^{2}) dt \le C.
\end{aligned}
\end{equation}
\end{lemma}

\begin{proof}1. Applying Lemma \ref{lem 2.3} with $\mathbf{F}=-\rho\boldsymbol{u}_{t}-\rho\boldsymbol{u}\cdot\nabla\boldsymbol{u}+\rho
\theta\boldsymbol{e_{2}}+\boldsymbol{b}\cdot\nabla\boldsymbol{b}$, it follows from  H\"older's inequality, \eqref{R1}, \eqref{R4}, \eqref{aa} and \eqref{R11} that
\begin{equation}
\begin{aligned}
\|\nabla^{2}\boldsymbol{u}\|_{L^{4}}&\le  C(\|\rho\boldsymbol{u}_{t}\|_{L^{4}}+\|\rho\boldsymbol{u}\cdot\nabla\boldsymbol{u}
\|_{L^{4}}+\|\rho\theta\|_
{L^{4}}+\|\boldsymbol{b}\cdot\nabla\boldsymbol{b}\|_{L^{4}})\\
&\le C\|\sqrt{\rho}\boldsymbol{u}_{t}\|_{L^{4}}
+C\|\boldsymbol{u}\|_{L^{8}}\|\nabla\boldsymbol{u}\|_{L^{8}}+C\|\theta\|_{L^{4}}
+C\|\boldsymbol{b}\|_{L^{8}}\|\nabla\boldsymbol{b}\|_{L^{8}}\\
&\le C\|\sqrt{\rho}
\boldsymbol{u}_{t}\|_{L^{2}}^{\frac{1}{4}}\|\sqrt{\rho}\boldsymbol{u}_{t}\|_{L^{6}}
^{\frac{3}{4}}+C(1+\|\nabla
\boldsymbol{u}\|_{L^{2}})\|\nabla\boldsymbol{u}\|_{L^{2}}^{\frac{1}{2}}\|\nabla^{2}
\boldsymbol{u}\|_{L^{4}}^{\frac{1}{2}}\\&\quad
+C(\|\sqrt{\rho}\theta\|_{L^{2}}+\|\nabla
\theta\|_{L^{2}})+C(\|\boldsymbol{b}\|_{L^{2}}+\|\nabla
\boldsymbol{b}\|_{L^{2}})(\|\nabla\boldsymbol{b}\|_{L^{2}}+\|\nabla^{2}\boldsymbol
{b}\|_{L^{2}})\\
&\le \frac{1}{2}\|\nabla^{2}\boldsymbol{u}\|_{L^{4}}+C\|\sqrt{\rho}\boldsymbol{u}_{t}\|_{L^{2}}^{\frac{1}{4}}
(\|\sqrt{\rho}\boldsymbol{u}_{t}\|_{L^{2}}+\|\nabla\boldsymbol
{u}_{t}\|_{L^{2}})^{\frac{3}{4}}+C\|\Delta\boldsymbol{b}
\|_{L^{2}}+C,
\end{aligned}
\end{equation}
which gives
\begin{equation}\label{R23}
\|\nabla^{2}\boldsymbol{u}\|_{L^{4}}\le C\|\sqrt{\rho}\boldsymbol{u}_{t}\|_{L^{2}}+\|\sqrt{\rho}\boldsymbol{u}_{t}\|_{L^{2}}
^{\frac{1}{4}}\|\nabla\boldsymbol{u}_{t}\|_{L^{2}}^{\frac{3}{4}}+
\|\Delta\boldsymbol{b}\|_{L^{2}}+C.
\end{equation}

In view of Sobolev's inequality, \eqref{R1} and \eqref{R23}, we obtain that
\begin{equation}\label{R24}
\begin{aligned}
& \int_{0}^{T}\|\nabla\boldsymbol{u}\|_{L^{\infty}}dt \\
\le &  C\int_{0}^{T}\|\nabla\boldsymbol{u}\|
_{W^{1,4}}dt\\
\le & C\int_{0}^{T}(\|\nabla\boldsymbol{u}\|_{L^{4}}+\|\nabla^{2}\boldsymbol{u}\|_
{L^{4}})dt\\
\le & C\int_{0}^{T}(\|\nabla\boldsymbol{u}\|_{L^{2}}+\|\nabla^{2}
\boldsymbol{u}\|_{L^{4}})dt\\
\le & C\int_{0}^{T}
(1+\|\nabla\boldsymbol{u}\|_{L^{2}}+\|\sqrt{\rho}\boldsymbol{u}_{t}\|_{L^{2}}
+\|\sqrt{\rho}\boldsymbol{u}
\|_{L^{2}}^{\frac{1}{4}}\|\nabla\boldsymbol{u}_{t}\|_{L^{2}}^{\frac{3}{4}}+
\|\Delta\boldsymbol{b}
\|_{L^{2}}^{2})dt\\
\le & C\int_{0}^{T}(t^{\frac{1}{2}}\|\sqrt{\rho}\boldsymbol{u}_{t}\|_{L^{2}})
^{\frac{1}{4}}(t^{\frac{1}{2}}\|\nabla\boldsymbol{u}_{t}\|_{L^{2}})^{\frac{3}{4}}
\cdot t^{\frac{1}{2}}dt +C\int_{0}^{T}\|\Delta\boldsymbol{b}\|_{L^{2}}dt+C\\
\le &  \sup_{[0,T]}(t\|\sqrt{\rho}
\boldsymbol{u}_{t}\|_{L^{2}}^{2})^{\frac{1}{8}}\left(t^{\frac{1}{2}}\|\nabla
\boldsymbol{u}_{t}\|_{L^{2}}^{2}dt
\right)^{\frac{3}{8}}\left(\int_{0}^{T}t^{-\frac{4}{5}}dt\right)^{\frac{5}{8}}
+C \\
\le & C.
\end{aligned}
\end{equation}
2. Taking spatial derivative $\nabla$ on the transport equation $\eqref{Benard 2}_{1}$ gives
\begin{equation}\label{R25}
\partial_{t}\nabla\rho+(\boldsymbol{u}\cdot\nabla)\nabla\rho+\nabla\boldsymbol{u}
\cdot\nabla\rho= \boldsymbol{0}.
\end{equation}
Multiplying \eqref{R25} by $p|\nabla\rho|^{p-2}\nabla\rho$ for $2\le p\le q$, it is easy to find that
\begin{equation}
\frac{d}{dt}\|\nabla\rho\|_{L^{p}}\le C\|\nabla\boldsymbol{u}\|_{L^{\infty}}\|\nabla\rho\|_{L^{p}},
\end{equation}
which combined with Gronwall's inequality and \eqref{R24} gives
\begin{equation}\label{R26}
\underset{[0,T]}{\sup}\|\nabla\rho\|_{L^{p}}\le C.
\end{equation}
This together with the earlier estimate \eqref{z2} leads to
\begin{equation}\label{R27}
\underset{[0,T]}{\sup}\|\rho-\tilde{\rho}\|_{H^{1}\cap W^{1,q}}\le C.
\end{equation}
In addition, for $2\le r<q$, we deduce from the transport equation $\eqref{Benard 2}_{1}$, H\"older's inequality and \eqref{R1} that
\begin{equation}
\|\rho_t\|_{L^r}=\|\boldsymbol{u}\cdot\nabla\rho\|_{L^{r}}\le \|\boldsymbol{u}\|_{L^{\frac{qr}{q-r}}}
\|\nabla\rho\|_{L^{q}}\le C\|\boldsymbol{u}\|_{H^{1}}\|\nabla\rho\|_{L^{q}},
\end{equation}
which together with $\eqref{R26}$ and \eqref{R11} yields
\begin{equation}\label{R28}
\sup_{[0, T]}\|\rho_{t}\|_{L^{r}}\le C.
\end{equation}
Here we cannot verify the case of $r=q$ since the Gagliardo-Nirenberg inequality \eqref{R1} fails for $p=\infty$ in the critical spatial dimension.

\noindent 3. We infer from \eqref{stokes 2}, \eqref{R7}, \eqref{R26}, \eqref{ll01}, \eqref{R11} and Sobolev's inequality that
\begin{equation}\label{R29}
\begin{aligned}
\|\nabla\boldsymbol{u}\|_{H^{2}}^{2}+\|\nabla P\|_{H^{1}}^{2}
\le & C(\|\rho\boldsymbol
{u}_{t}\|_{H^{1}}^{2}+\|\rho\boldsymbol{u}\cdot\nabla\boldsymbol{u}\|_{H^{1}}^{2}+\|\rho\theta\|_{H^{1}}^{2}+\|\boldsymbol{b}\cdot\nabla\boldsymbol{b}\|_{H^{1}}^{2} + \|\nabla \boldsymbol{u}\|_{H^1}^2) \\
\le &  C\|\rho\boldsymbol{u}_{t}\|_{L^{2}}^{2}+ C\|\nabla(\rho\boldsymbol{u}_{t})\|
_{L^{2}}^{2}+C\|\rho\boldsymbol{u}\cdot\nabla\boldsymbol{u}\|_{L^{2}}^{2}
+C\|\nabla
(\rho\boldsymbol{u}\cdot\nabla\boldsymbol{u})\|_{L^{2}}^{2}\\&+C\|\rho\theta\|_{L^{2}
}+C\|\nabla(\rho\theta)\|_{L^{2}}+C\|\boldsymbol{b}\cdot\nabla\boldsymbol{b}\|
_{L^{2}}^{2}+C\|\nabla(\boldsymbol{b}\cdot\nabla\boldsymbol{b})\|_{L^{2}}^{2}\\
\le &
C\|\sqrt{\rho}\boldsymbol{u}_{t}\|_{L^{2}}^{2}+C\|\nabla\rho\|_{L^{q}}^{2}\|
\boldsymbol{u}_{t}\|_{L^{\frac{2q}{q-2}}}^{2}+C\|\nabla\boldsymbol{u}_{t}\|_{L^{2}}
^{2}+C\|\nabla\boldsymbol{b}\|_{L^{4}}^{4}\\&+C\|\boldsymbol{u}\|_{L^{\infty}}^{2}\|\nabla\boldsymbol{u}\|_{L^{2}}^{2}+
C\|\boldsymbol{u}\|_{L^{\infty}}^{2}\|\nabla^{2}\boldsymbol{u}\|_{L^{2}}^{2}+C
\|\nabla\boldsymbol{u}\|_{L^{4}}^{4}+C\|\nabla\theta\|_{L^{2}}^{2}\\
&+C\|\nabla\rho\|_{L^{q}}^{2}\|\boldsymbol{u}\|
_{L^{\frac{2q}{q-2}}}^{2}\|\nabla\boldsymbol{u}\|_{L^{2}}^{2}+C\|\sqrt{\rho}\theta\|
_{L^{2}}^{2}+C\|\nabla\rho\|_{L^{q}}^{2}\|\theta\|
_{L^{\frac{2q}{q-2}}}^{2}\\
&+C\|\boldsymbol{b}\|_{L^{\infty}}^{2}\|\nabla
\boldsymbol{u}\|_{L^{2}}^{2}
+C\|\boldsymbol{b}\|_{L^{\infty}}^{2}\|
\nabla^{2}\boldsymbol{b}
\|_{L^{2}}^{2}+C\|\nabla\boldsymbol{u}\|_{L^{2}}^{2}\\
\le& C\|\sqrt{\rho}\boldsymbol{u}_{t}\|_{L^{2}}^{4}+C\|\nabla\boldsymbol{u}_{t}\|
_{L^{2}}^{4}+C\|\Delta\boldsymbol{b}\|_{L^{2}}^{4}+C,
\end{aligned}
\end{equation}
which combined with the obtained estimates \eqref{R11} and \eqref{Q11} yields
\begin{equation}\label{R30}
\int_{0}^{T}t\|\nabla\boldsymbol{u}\|_{H^{2}}^{2}dt\le C.
\end{equation}

In a similar way, combining the well-known regularity theory of the elliptic equations for $\theta$ and $\boldsymbol{b}$, we obtain that
\begin{equation}\label{R31}
\int_{0}^{T}\left(t\|\nabla\boldsymbol{b}\|_{H^{2}}^{2}+t\|\nabla\theta\|_{H^{2}}^{2}\right) dt \le C.
\end{equation}
Therefore, we finish the proof of Lemma \ref{lemma 3.4}.\\
\end{proof}

\section{proof of theorem 1.1}\label{section 4}
With Lemma \ref{local} and all the \textit{a priori} estimates obtained in Section 3 at hand, we shall give a proof of Theorem \ref{main}.

In view of Lemma \ref{local}, there exists a $T_{*}>0$ such that the problem $\eqref{Benard}$-$\eqref{initial}$ has a unique local strong solution $(\rho,\boldsymbol{u},\theta,\boldsymbol{b},P)$ on $\mathbb{R}^{2}\times (0,T_{*}]$. We plan to extend the local solution to global in time.

Set
\begin{equation}\label{H1}
T^{*}\triangleq \sup\left\{T|(\rho,\boldsymbol{u},\theta,\boldsymbol{b},P)~\text{is a strong
  solution on}~\mathbb{R}^{2}\times (0,T]\right\}.
\end{equation}
First, for $T_* < T \le T^*$ with $T$ finite, one deduces from that
$$ \nabla \boldsymbol{u}, \nabla \theta, \nabla \boldsymbol{b} \in C([\tau, T]; H^1),$$
where we used the standard embedding
$$ L^{\infty}(\tau, T; H^2) \cap H^1(\tau, T; L^2) \hookrightarrow C([\tau, T]; H^1).$$

And it follows from $\eqref{R27}$ and $\eqref{R28}$ that
\begin{equation}\label{H2}
\rho-\widetilde{\rho}\in C([0,T];H^{1}\cap W^{1,q}).
\end{equation}
Owing to $\eqref{R7}$ and $\eqref{aa}$, we have
\begin{equation}
\rho\boldsymbol{u}_{t}=\sqrt{\rho}\cdot\sqrt{\rho}\boldsymbol{u}_{t}\in L^{2}(0,T;L^{2}).
\end{equation}
By $\eqref{R28}$ and the Sobolev's inequality, we get
\begin{equation}
\rho_{t}\boldsymbol{u}\in L^{\infty}(0,T;L^{2}).
\end{equation}
Thus, we arrive at
\begin{equation}\label{H3}
(\rho\boldsymbol{u})_{t}=\rho\boldsymbol{u}_{t}+\rho_{t}\boldsymbol{u}\in L^{2}(0,T;L^{2}),
\end{equation}
which combined with $\rho\boldsymbol{u}\in L^{\infty}(0,T;L^{2})$ due to $\eqref{aa}$ and $\eqref{R7}$ gives rise to
\begin{equation}\label{H4}
\rho\boldsymbol{u}\in C(0,T;L^{2}).
\end{equation}
Similarly, one has
\begin{equation}
\rho\theta\in C(0,T;L^{2}),\boldsymbol{b}\in C(0,T;L^{2}).
\end{equation}
Finally, if $T^{*} <  \infty$, we can verify that
\begin{equation}
(\rho,\boldsymbol{u},\theta,\boldsymbol{b})(x,T^{*})=\lim_{t\to T^{*}}(\rho,\boldsymbol{u},\theta,\boldsymbol{b})(x,T)
\end{equation}
satisfies the initial condition $\eqref{initial 2}$ at $t=T^{*}$. Thus, taking $(\rho,\boldsymbol{u},\theta,\boldsymbol{b})(x,T^{*})$ as the initial data, Lemma \ref{local} implies that one can extend the strong solution beyond $T^{*}$. This contradicts the assumption of $T^{*}$ in \eqref{H1}. Furthermore, the estimates as those in $\eqref{ee}$ follow from Lemmas \ref{lem 3.1}-
\ref{lemma 3.4}. This completes the proof of Theorem \ref{main}.

\end{document}